\documentclass{amsart}
\usepackage{enumerate,amssymb,amsmath,graphics,epsfig}
\usepackage{color}
\usepackage{xspace}
\usepackage{mathrsfs}
\usepackage{arydshln}
\usepackage[labelfont=rm,font=small,labelformat=simple]{subcaption}

\newcommand\RE{\mathbb{R}}

\renewcommand\div{\operatorname{\mathrm{div}}}
\newcommand\divs{\div_{\bfs}}

\newcommand\Grad{\operatorname{\boldsymbol\nabla}}
\newcommand\Grads{\operatorname{\boldsymbol\nabla}_{\bfs}}
\newcommand\symgrad{\operatorname{\underline{\boldsymbol\varepsilon}}}
\newcommand\Huo{\bfH^1_0(\Omega)}
\newcommand\Hub{\bfH^1(\B)}

\newcommand\Ldo{L^2_0(\Omega)}
\newcommand\Ldb{\bfL^2(\B)}

\newcommand\Of{\Omega^f}
\newcommand\Oft{\Omega^f_t}
\newcommand\Os{\Omega^s}
\newcommand\Ost{\Omega^s_t}
\newcommand\B{\mathcal{B}}
\newcommand\bfH{\mathbf{H}}
\newcommand\bfL{\mathbf{L}}
\newcommand\bfU{\mathbf{U}}
\newcommand\bfV{\mathbf{V}}
\newcommand\bfS{\mathbf{S}}
\newcommand\bfW{\mathbf{W}}
\newcommand\bfX{\mathbf{X}}
\newcommand\bfx{\mathbf{x}}
\newcommand\bfy{\mathbf{y}}
\newcommand\bfs{\mathbf{s}}
\newcommand\bfu{\mathbf{u}}
\newcommand\bfv{\mathbf{v}}
\newcommand\bfw{\mathbf{w}}
\newcommand\bfz{\mathbf{Y}}
\newcommand\bfn{\mathbf{n}}
\newcommand\bfd{\mathbf{d}}
\newcommand\bff{\mathbf{f}}
\newcommand\bfg{\mathbf{g}}
\newcommand\bfLambda{\boldsymbol{\Lambda}}
\newcommand\bflambda{\boldsymbol{\lambda}}
\newcommand\bfmu{\boldsymbol{\mu}}
\newcommand\bfsigma{\boldsymbol{\sigma}}
\newcommand\OT{0,T}
\newcommand\OO{{0,\Omega}}
\newcommand\OB{{0,\B}}

\newcommand\Xb{\overline\bfX}

\newcommand\bbF{\mathbb{F}}
\newcommand\bbP{\mathbb{P}}
\newcommand\bbA{\mathbb{A}}
\newcommand\bbB{\mathbb{B}}
\newcommand\bbK{\mathbb{K}}
\newcommand\bbV{\mathbb{V}}
\newcommand\ucX{\bfu(\bfX(\cdot,t),t)}
\newcommand\vcX{\bfv(\bfX(\cdot,t))}
\newcommand\ucXb{\bfu(\Xb(\cdot,t),t)}
\newcommand\vcXb{\bfv(\Xb(\cdot,t))}
\newcommand\dr{\delta_\rho}
\newcommand\Ho{\bfH_0}
\newcommand\Vo{\mathbf{V}_0}
\newcommand\dt{\Delta t}
\newcommand\m[1]{\mathsf{#1}}
\newcommand\calA{\mathscr{A}}
\newcommand\calB{\mathscr{B}}
\newcommand\T{\mathcal{T}}
\newcommand\hx{h_\bfx}
\newcommand\hs{h_\bfs}
\newcommand\hl{h_{\bflambda}}

\newcommand\Xoh{\mathring\bfX_h}

\newcommand\CNM{\textsf{CNm}\xspace}
\newcommand\CNT{\textsf{CNt}\xspace}
\newcommand\BDF{\textsf{BDF2}\xspace}
\newcommand\BD{\textsf{BDF1}\xspace}
\theoremstyle{plain}
\newtheorem{thm}{Theorem}
\newtheorem{proposition}[thm]{Proposition}

\newtheorem{problem}{Problem}
\newtheorem{ass}{Assumption}
\theoremstyle{remark}

\newtheorem{remark}{Remark}

\begin{document}
\title[]
{Existence, uniqueness, and approximation of a fictitious domain formulation
for fluid-structure interactions}
\author{Daniele Boffi}
\address{King Abdullah University of Science and Technology (KAUST), Saudi
Arabia and Universit\`a degli Studi di Pavia, Italy}
\email{daniele.boffi@kaust.edu.sa}
\urladdr{https://cemse.kaust.edu.sa/people/person/daniele-boffi}
\author{Lucia Gastaldi}
\address{DICATAM, Universit\`a degli Studi di Brescia, Italy}
\email{lucia.gastaldi@unibs.it}
\urladdr{http://lucia-gastaldi.unibs.it}

\begin{abstract}
In this paper we describe a computational model for the simulation of
fluid-structure interaction problems based on a fictitious domain approach. We
summarize the results presented over the last years when our research evolved
from the Finite Element Immersed Boundary Method (FE-IBM) to the actual Finite
Element Distributed Lagrange Multiplier method (FE-DLM). We recall the
well-posedness of our formulation at the continuous level in a simplified
setting. We describe various time semi-discretizations that provide
unconditionally stable schemes.  Finally we report the stability analysis for
the finite element space discretization where some improvements and
generalizations of the previous results are obtained.
\end{abstract}
\maketitle
\section{Introduction}
\label{se:intro}

In this paper we summarize in a unified setting some results of our research on
the modeling and the approximation of fluid-structure interaction problems. Our
aim is to describe the dynamics of a solid elastic body immersed in a Newtonian
incompressible fluid. Here, we consider the so called zero-codimension case,
that is the solid and the fluid are both two- or three-dimensional.
From the mathematical point of view, the interaction is described by different
partial differential equations in the regions occupied by the fluid and the
solid, coupled with suitable transmission conditions along the interface between
the two.
It is well known that the numerical approximation of fluid-structure
interaction problems is challenging for several reasons: first of all the
numerical method must track the movement of the structure and the corresponding
computational grids should allow the evaluation of quantities defined on moving
domains. In this context, the use of a Lagrangian framework is more suited for
the simulation of the structure deformation, while the approximation of the
fluid velocity and pressure is better performed by an Eulerian approach.

Another crucial issue related to the approximation of fluid-structure
interactions is how to deal with the coupling of the two underlying models:
monolithic approaches perform the simultaneous computation of the fluid and
structure unknowns, while partitioned schemes combine different solvers in the
two subregions with an iterative procedure. In general monolithic schemes
require implicit nonlinear solvers and a careful trade off between superior
stability properties and more demanding computational load.

The research in this framework is very active and is based on a wide literature,
ranging from boundary fitted approaches which, typically, use the so called
Arbitrary Lagrangian Eulerian method~\cite{hirt,Donea1977,HLZ,doneahuerta2004}
to non fitted approaches which include, for instance, level set
methods~\cite{levelset} Nitsche and XFEM methods~\cite{nitsche,XFEM}.
Our model belongs to the latter family originating from the Immersed
Boundary Method (IBM)~\cite{peskin,BGHP} and evolved towards a fictitious
domain approach in the spirit
of~\cite{glopanper1,glopanper2,girglo1995,glopanhj,girglopan,yu}.
Obviously, no method is the optimal choice for all cases and, depending on the
particular situation, it could be preferable to make use of different
approaches; our formulation has the advantage to be unconditionally stable in
time~\cite{BCG,wolf} without the need of using fully implicit time schemes
and, being based on non fitted meshes, can accommodate larger displacements.
On the other hand, the coupling between fluid and structure models requires
the evaluation of integrals that combine basis functions defined on different
meshes.
A solid mathematical analysis has been performed; we shall review some of the
results in the following sections giving reference to the original papers when
appropriate. Moreover, we extend the discretization of our model, allowing for
more general choices of finite element spaces.
We describe an incompressible solid immersed in an incompressible fluid; more
general situation could be considered, involving compressible
solids~\cite{compressible}.

In Section~\ref{se:model} we recall the problem we are interested in, and
introduce our fictitious domain formulation. Next, we analyze the continuous
problem in Section~\ref{se:existence} in a linearized setting, assuming that
the motion of the solid is prescribed. Section~\ref{se:dt} deals with the time
discretization; the main result of this section is the unconditional stability
of the evolution scheme. The space discretization is considered in
Section~\ref{se:saddle} where a stability analysis is presented which leads to
optimal convergence estimates for the steady state solution.
Finally, Section~\ref{se:numerical} reports on several numerical tests that
confirm the good behavior of our approach.

\section{Model problem and fictitious domain formulation}
\label{se:model}

The problem we want to address is easily explained in the following
simplified setting. We consider a solid immersed in a fluid in two or three
dimensions. At time $t$ the solid is located in the region $\Ost\subset\RE^d$
($d=2,3$) which is the image of a reference configuration $\B$ through a
mapping $\bfX:\B\to\RE^d$. The fluid occupies the region $\Oft\subset\RE^d$ so
that we are interested in a dynamic occurring in the union of $\Ost$ and
$\Oft$. A typical assumption is that, denoting by $\Omega$ the interior of the
union of the closures of $\Ost$ and $\Oft$, then $\Omega$ does not depend on
$t$. This assumption is reasonable for several applications; in general
$\Omega$ can be thought as a container where the dynamics takes place: for
instance, the solid can be inside the fluid and far away from the exterior
boundary of it, or the solid can touch one fixed part of the container. In
this paper we deal with the first situation.
We denote by $\Gamma_t$ the interface between fluid and solid, which can be
defined as the interior of the intersection of $\overline{\Ost}$ and
$\overline{\Oft}$.

The system is described by the fluid velocity $\bfu^f$ and pressure $p^f$, and
by the solid position $\bfX$. The velocity and the pressure depend on time and
on the space Eulerian variable $\bfx\in\Oft$, while the position $\bfX$
depends on time and on the Lagrangian variable $\bfs\in\B$.
In the fixed domain $\Omega$ we are using the Eulerian framework and the
corresponding variable $\bfx$. A point $\bfx$ of the domain $\Ost$ can be
expressed at time $t$ in the Lagrangian setting as
\[
\bfx=\bfX(\bfs,t).
\]
The kinematic condition is expressed by the following relationship between the
material velocity $\bfu^s$ and $\bfX$:
\[
\bfu^s(\bfx,t)=\frac{\partial\bfX}{\partial t}(\bfs,t),
\]
where $\bfx=\bfX(\bfs,t)$. The deformation gradient is given by
\[
\bbF(\bfs,t)=\frac{\partial\bfX}{\partial\bfs}(\bfs,t).
\]
We denote by $|\bbF|$ its determinant. We consider an incompressible solid, so
that $|\bbF|$ is constant in time; in particular, in the case when $\B$ is the
initial configuration $\Os_0$ of $\Ost$, we have $|\bbF|=1$.

In the incompressible fluid the Navier--Stokes equations describe the dynamics
as follows
\begin{equation}
\aligned
&\rho_f\left(\frac{\partial\bfu^f}{\partial t}+\bfu^f\cdot\Grad\bfu^f\right)=
\div\bfsigma^f&&\text{in }\Oft\\
&\div\bfu^f=0&&\text{in }\Oft,
\endaligned
\label{eq:ns}
\end{equation}
where $\rho_f$ is the fluid density and $\bfsigma^f$ is the Cauchy stress
tensor that reads
\[
\bfsigma^f=-p^f{\mathbb I}+\nu_f\symgrad(\bfu^f),
\]
$\nu_f>0$ being the viscosity of the fluid and $\symgrad$ the symmetric
gradient.

We assume an incompressible viscoelastic material that can be described by a
Cauchy stress tensor composed of two parts
$\bfsigma^s=\bfsigma^s_f+\bfsigma^s_s$: the first one is analogous to the
fluid stress with the introduction of an artificial pressure $p^s$, which is
the Lagrange multiplier associated with the incompressibility,
\[
\bfsigma^s_f=-p^s{\mathbb I}+\nu_s\symgrad(\bfu^s),
\]
$\nu_s>0$ being the body viscosity; the second term is related to the
Piola--Kirchhoff elasticity stress tensor $\bbP$ via the Piola transformation
\[
\bfsigma^s_s=|\bbF|^{-1}\bbP\bbF^\top.
\]
The elastic part of the stress can be modeled using a potential energy density
$W(\bbF,\bfs,t)$ so that
\[
\bbP(\bbF,\bfs,t)=\frac{\partial W}{\partial\bbF}(\bbF,\bfs,t).
\]
Taking all this into account, the equations describing the solid are
\begin{equation}
\aligned
&\rho_s\frac{\partial^2\bfX}{\partial t^2}
=\divs(|\bbF|\bfsigma_f^s\bbF^{-\top}+\bbP(\bbF))\ &&\text{in }\B\\
&\div\bfu^s=0&&\text{in }\Ost,
\endaligned
\label{eq:ela}
\end{equation}
where $\rho_s$ is the solid density. The description of the model requires
suitable transmission conditions enforcing the appropriate continuities of the
velocity and of the Cauchy stress across the interface
$\Gamma_t$ which can be stated as follows
\begin{equation}
\aligned
&\bfu^f=\bfu^s&&\text{on }\partial\Ost\\
&\bfsigma^f\bfn_f=-(\bfsigma^s_f+|\bbF|^{-1}\bbP\bbF^\top)\bfn_s
&&\text{on }\partial\Ost,
\endaligned
\label{eq:trans}
\end{equation}
where $\bfn_f$ and $\bfn_s$ stand for the outward unit normal vectors to
$\Oft$ and $\Ost$, respectively. In conclusion, the system is described
by~\eqref{eq:ns}, \eqref{eq:ela}, \eqref{eq:trans}, and the following initial
and boundary conditions
\begin{equation}
\aligned
&\bfu^f(0)=\bfu^f_0&&\text{in }\Of_0\\
&\bfu^s(0)=\bfu^s_0&&\text{in }\Os_0\\
&\bfX(0)=\bfX_0&&\text{in }\B\\
&\bfu^f=0&&\text{on }\partial\Omega.
\endaligned
\label{eq:IBC}
\end{equation}

Before describing our variational formulation we recall some standard notation
that we are going to adopt~\cite{LM}. Given a domain $D$, the space
$\mathscr{D}(D)$ is the space of infinitely differentiable functions with
compact support in $D$, $L^2(D)$ is the space of square integrable functions
on $D$, the standard Sobolev spaces are denoted by $W^{s,p}(D)$, where
$s\in\RE$ refers to the differentiability and $p\in[1,+\infty]$ to the
integrability exponent. As usual, when $p=2$ we use the notation $H^s(D)$. The
corresponding norm is indicated by $\|\cdot\|_{s,D}$ and the scalar product in
$L^2(D)$ by $(\cdot,\cdot)_D$; when no confusion arises we omit the indication
of the domain $D$. In particular we will usually omit $\Omega$, while we will
indicate explicitly when quantities are defined on the domain $\B$.
$L^2_0(D)$ stands for the subspace of zero mean valued functions and
$H^1_0(D)$ is the subset of functions in $H^1(D)$ with zero trace on
$\partial D$.
Given Banach spaces $X$ and $Y$, the notation $Y(\OT;X)$ contains space-time
functions that for almost all $t\in\OT$ are in $X$ and that are in $Y$ as
functions from $(\OT)$ to $X$.
Functional spaces of vector valued functions are indicated with boldface
letters.

The main idea behind the fictitious domain approach that we are going to
adopt, consists in extending the fluid variables inside the solid domain so
that all involved quantities are defined in $\Omega$ (Eulerian variables) or
$\B$ (Lagrangian variables).
We started considering a fictitious domain model for a simplified interface
problem~\cite{risulti,ruggeri} which has been extended to fluid-structure
interactions in~\cite{BCG}.

We denote by $\bfu$ and $p$ the velocity and
pressure in $\Omega$, with the understanding that their restrictions to the
two subdomains $\Oft$ and $\Ost$ coincide with $\bfu^f$, $p^f$ and $\bfu^s$,
$p^s$, respectively. With the aim of presenting a variational formulation of
our problem, the condition $\bfu|_{\Ost}=\bfu^s$ will be enforced with the
help of a bilinear form.
Let $\bfLambda$ be a Hilbert space and $c:\bfLambda\times\Hub\to\RE$ a
continuous bilinear form with the property
\[
c(\bfmu,\bfz)=0\quad\forall\mu\in\bfLambda\quad\text{implies}\quad
\bfz=0.
\]

The variational formulation is described by making use of the following
notation.
\[
\aligned
&\nu=
\begin{cases}
\nu_f&\text{in }\Oft\\
\nu_s&\text{in }\Ost
\end{cases}\\
&a(\bfu,\bfv)=\int_\Omega\nu\symgrad(\bfu):\symgrad(\bfv)\,d\bfx\\
&b(\bfu,\bfv,\bfw)=
\int_\Omega\frac{\rho_f}2\left((\bfu\cdot\Grad\bfv)\cdot\bfw
-(\bfu\cdot\Grad\bfw)\cdot\bfv\right)\,d\bfx\\
&\dr=\rho_s-\rho_f.
\endaligned
\]

\begin{problem}[Fictitious domain formulation]
Given $\bfu_0\in\Huo$, $\bfX_0\in\bfW^{1,\infty}(\B)$, and $\bfX_1\in\Hub$,
find $\bfu(t)\in\Huo$, $p(t)\in\Ldo$, $\bfX(t)\in\Hub$, and
$\bflambda(t)\in\bfLambda$ such that, for almost every $t\in(\OT)$, it holds
\begin{equation}
\aligned
&\rho_f\left(\frac{\partial\bfu}{\partial t}(t),\bfv\right)
   +b(\bfu(t),\bfu(t),\bfv)+a(\bfu(t),\bfv)\\
&\qquad-(\div\bfv,p(t))+c\left(\bflambda(t),\vcX\right)=0
   &&\forall\bfv\in\Huo\\
&(\div\bfu(t),q)=0&&\forall q\in\Ldo\\
&\dr\left(\frac{\partial^2\bfX}{\partial t^2}(t),\bfz\right)_{\B}
   +\left(\bbP(\bbF(t)),\Grads\bfz\right)_{\B}
   -c\left(\bflambda(t),\bfz\right)=0
&&\forall\bfz\in\Hub\\
&c\left(\bfmu,\ucX-\frac{\partial\bfX}{\partial t}(t)\right)=0 
   &&\forall\bfmu\in\bfLambda\\
&\aligned
&\bfu(0)=\bfu_0&&\text{in }\Omega\\
&\bfX(0)=\bfX_0&&\text{in }\B\\
&\frac{\partial\bfX}{\partial t}(0)=\bfX_1&&\text{in }\B.
\endaligned
\endaligned
\label{eq:DLM}
\end{equation}

\label{pb:DLM}
\end{problem}

\begin{remark}

The initial condition $\bfX_1$ in Problem~\ref{pb:DLM} is related to
$\bfu_0^s$ of~\ref{eq:IBC} by the relation
\[
\bfX_1=\bfu_0^s(\bfX_0)\qquad\text{in $\B$}.
\]

\end{remark}

Various choices have been presented for the bilinear form $c$ responsible for
the coupling of the Lagrangian and Eulerian frames.

In our setting two possible definitions of $c$ have been discussed
in~\cite{BCG,BG}: a natural choice is to consider as $\bfLambda$ the dual
space of $\Hub$ so that $c$ can
be taken as the duality pairing that certainly satisfies the required
properties; a second equivalent choice stems from interpreting the duality
pairing as the scalar product in $\Hub$ by the Riesz representation theorem so
that $\bfLambda=\Hub$. More in detail, we have the following definitions

\begin{enumerate}[1.]

\item
$\bfLambda_1=\Hub'$ and $c_1:\bfLambda_1\times\Hub\to\RE$ with
\begin{equation}
c_1(\bfmu,\bfz)={}_{\bfLambda_1}\langle\bfmu,\bfz\rangle_{\Hub}
\label{eq:choice1}
\end{equation}

\item
$\bfLambda_2=\Hub$ and $c_2:\bfLambda_2\times\Hub\to\RE$ with
\begin{equation}
c_2(\bfmu,\bfz)=(\bfmu,\bfz)_\B+(\Grads\bfmu,\Grads\bfz)_\B.
\label{eq:choice2}
\end{equation}

\end{enumerate}

While the two definitions are equivalent for the continuous problem, they give
rise to different discretizations. In the sequel we are going to use the
generic notation $\bfLambda$ and $c$, while indicating explicitly one of the
two cases when needed.

An analogous formulation, which is outside the topics of the present work, can
also be used in the case of codimension one structures.
We refer the interested reader to~\cite{BCG,BG}.

We end this section by stating a stability result for the continuous problem
which was proved in~\cite{BCG}.

\begin{proposition}
Let $\bfu(t)\in\Huo$ and $\bfX(t)\in\Hub$ be solutions of
Problem~\ref{pb:DLM}. Assume that $\partial\bfX(t)/\partial t\in\bfL^2(\B)$
and consider the elastic potential energy of the body given by
\[
E\left(\bfX(t)\right)=\int_\B W(\bbF(\bfs,t))\,d\bfs.
\]
Then the following conservation property is satisfied for almost every
$t\in(\OT)$
\[
\frac{\rho_f}2\frac{d}{dt}||\bfu(t)||^2_{\OO}
+||\nu^{1/2}\symgrad(\bfu(t))||^2_{\OO}+
\frac{\dr}2\frac{d}{dt}\left\|\frac{\partial\bfX}{\partial t}(t)\right\|^2_{\OB}
+\frac{d}{dt}E(\bfX(t))=0.
\]
\end{proposition}

\section{Existence and uniqueness of the linearized problem}
\label{se:existence}

Not many results are available in the literature about existence and
uniqueness of the solution to fluid-structure interaction problems. This is
not surprising since the coupling between fluids and solids gives rise in
general to highly non linear problems. In the case when a fluid is containing
rigid solids or elastic bodies described by a finite number of modes,
existence and uniqueness of weak solutions have been studied for instance
in~\cite{CSMT,DE1999,DE2000,DEGLT,F,GM,GLS,HS,Serre,Ta,TaTu}; when a fluid is
enclosed in a solid membrane then the existence and uniqueness of weak
solutions have been discussed in~\cite{Beirao,CDEG,MC2013b,MC2016}. Moreover,
local-in-time existence and uniqueness of strong solutions for an elastic
structure immersed in a fluid are proved
in~\cite{Coutand2005,Coutand2006,RaymondVanni2014,Boulakia2017,Boulakia2019}.

In this section we describe the analysis performed in~\cite{existence} about
the existence and the uniqueness of a linearization of Problem~\ref{pb:DLM}
in the case when $\bfLambda=\Hub$ and the bilinear form $c$ is equal to the
scalar product in $\Hub$.
This is a first step towards the analysis of the full problem which could make
use of some fixed point strategy.

We consider a given function $\Xb$ that describes the motion of the solid. We
assume that $\Xb$ belongs to $C^1([\OT];\bfW^{1,\infty}(\B))$, is invertible
with Lipschitz inverse, and coincides with the identity at time $t=0$, that is
$\Xb(\bfs,0)=\bfs$. Moreover, we assume that the motion of the solid is
compatible with the incompressibility constraint, that is
$\det(\Grads\Xb(t))=1$ for all $t$.

We choose a linear model for the elasticity, namely $\bbP(\bbF)=\kappa\bbF$;
moreover, we introduce a new variable $\bfw(t)$ equal to the velocity of the
solid $\partial\bfX(t)/\partial t$, so that, after neglecting the convective
term in the Navier--Stokes equation, we are led to the following problem.

\begin{problem}[Linearized formulation]
Let us assume that $\Xb\in C^1([\OT];\bfW^{1,\infty}(\B))$ satisfies the
hypotheses described above.
Given $\bfu_0\in\Huo$, $\bfX_0\in\bfW^{1,\infty}(\B)$, and $\bfX_1\in\Hub$,
find $\bfu(t)\in\Huo$, $p(t)\in\Ldo$, $\bfX(t)\in\Hub$, $\bfw\in\Hub$, and
$\bflambda(t)\in\Hub$ such that, for almost every $t\in(\OT)$, it holds
\begin{equation}
\aligned
&\rho_f\left(\frac{\partial\bfu}{\partial t}(t),\bfv\right)
  +a(\bfu(t),\bfv)-(\div\bfv,p(t))\\
&\qquad\qquad\qquad\qquad\qquad+c(\bflambda(t),\vcXb)=0&&\forall\bfv\in\Huo\\
&(\div\bfu(t),q)=0&&\forall q\in\Ldo\\
&\dr\left(\frac{\partial\bfw}{\partial t}(t),\bfz\right)_{\B}
  +\kappa(\Grads\bfX(t),\Grads\bfz)_{\B}-c(\bflambda(t),\bfz)=0
  &&\forall\bfz\in\Hub\\
&\left(\frac{\partial\bfX}{\partial t}(t),\bfy\right)_{\B}=(\bfw(t),\bfy)_{\B} 
  &&\forall\bfy\in\Ldb\\
&c\left(\bfmu,\ucXb-\bfw(t)\right)=0&&\forall\bfmu\in\Hub\\
&\aligned
&\bfu(0)=\bfu_0&&\text{in }\Omega\\
&\bfX(0)=\bfX_0&&\text{in }\B\\
&\bfw(0)=\bfX_1&&\text{in }\B.
\endaligned
\endaligned
\label{eq:DLM-lin}
\end{equation}
\label{pb:DLM-lin}
\end{problem}

The following existence and uniqueness result was proved in~\cite{existence}.

\begin{thm}
Under the assumptions reported above, there exists a unique solution to
Problem~\ref{pb:DLM-lin} that satisfies the following regularity
\[
\aligned
&\bfu\in\bfL^\infty(\OT;\Ho)\cap\bfL^2(\OT;\Vo)\\
&p\in L^2(\OT;\Ldo)\\
&\bfX\in\bfL^\infty(\OT;\Hub)\\
&\bfw\in\bfL^\infty(\OT;\bfL(\B))\cap\bfL^2(\OT;\Hub)\\
&\bflambda\in\bfL^2(\OT;\Hub),
\endaligned
\]
where
\[
\aligned
&\mathscr{V}_0=\{\bfv\in\mathscr{D}(\Omega)^d: \div\bfv=0\}\\
&\Ho=\text{ the closure of }\mathscr{V}_0\text{ in }\Ldo\\
&\Vo=\text{ the closure of }\mathscr{V}_0\text{ in }\Huo.
\endaligned
\]
\end{thm}

The proof of this result is obtained by considering first a reduced problem
where the unknowns $p$ and $\bflambda$ are eliminated since the velocity is
sought in the kernel of the divergence operator $\Vo$ and the pair
$(\bfu(t),\bfw(t))$ is required to satisfy the constraint
\begin{equation}
c\left(\bfmu,\ucXb-\bfw(t)\right)=0\quad\forall\bfmu\in\Hub.
\label{eq:kernel}
\end{equation}
In this setting, the proof follows a suitable modification of the Galerkin
arguments used in~\cite{temam} for the analysis of Navier--Stokes equations.

Finally, the Lagrange multiplier and the pressure are recovered by using
Lax--Milgram lemma and the Banach closed range theorem.

\section{Time advancing schemes}
\label{se:dt}

We begin in this section the study of the numerical approximation of
Problem~\ref{pb:DLM}, starting from the time discretization.

Let us introduce a time discretization parameter $\dt$, and let us denote by
$t_n$, $n=0,\dots,N$ the corresponding nodes; the following system is obtained
by the application of the backward Euler scheme.

\begin{problem}[Backward Euler scheme]
Given $\bfu_0\in\Huo$, $\bfX_0\in\bfW^{1,\infty}(\B)$, and $\bfX_1\in\Hub$,
for all $n=1,\dots,N$ find $\bfu^n\in\Huo$, $p^n\in\Ldo$, $\bfX^n\in\Hub$,
and $\bflambda^n\in\bfLambda$ such that
\begin{equation}
\aligned
&\rho_f\left(\frac{\bfu^{n+1}-\bfu^n}{\dt},\bfv\right)
+b\left(\bfu^{n+1},\bfu^{n+1},\bfv\right)+a\left(\bfu^{n+1},\bfv\right)\\
&\hspace{2.05cm}-\left(\div\bfv,p^{n+1}\right)+
c\left(\bflambda^{n+1},\bfv(\bfX^{n+1})\right)=0&&\forall\bfv\in\Huo\\
&\left(\div\bfu^{n+1},q\right)=0&&\forall q\in\Ldo\\
&\dr\left(\frac{\bfX^{n+1}-2\bfX^n+\bfX^{n-1}}{\dt^2},\bfz\right)_\B
+\left(\bbP(\bbF^{n+1}),\Grads\bfz\right)_\B\\
&\hspace{5.5cm}-c\left(\bflambda^{n+1},\bfz\right)=0&&\forall\bfz\in\Hub\\
&c\left(\bfmu,\bfu^{n+1}(\bfX^{n+1})-\frac{\bfX^{n+1}-\bfX^n}{\dt}\right)=0
&&\forall\bfmu\in\bfLambda\\
&\aligned
&\bfu^0=\bfu_0&&\text{in }\Omega\\
&\bfX^0=\bfX_0&&\text{in }\B,
\endaligned
\endaligned
\label{eq:BE}
\end{equation}
where $\bfX^{-1}$ can be defined, for instance, from the following equation
\[
\frac{\bfX^0-\bfX^{-1}}{\dt}=\bfX_1\qquad\text{in }\B.
\]
\label{pb:BE}
\end{problem}

In~\cite{BCG} the following stability estimate was proved for the time
discretization presented in Problem~\ref{pb:BE}
\[
\aligned
&\frac{\rho_f}{2\dt}\left(\|\bfu^{n+1}\|_{\OO}^2-\|\bfu^n\|_{\OO}^2\right)
+\nu\|\symgrad\bfu^{n+1}\|_{\OO}^2+\frac{E(\bfX^{n+1})-E(\bfX^n)}{\dt}\\
&\qquad
+\frac{\dr}{2\dt}\left(\left\|\frac{\bfX^{n+1}-\bfX^n}{\dt}\right\|_{\OB}^2
-\left\|\frac{\bfX^n-\bfX^{n-1}}{\dt}\right\|_{\OB}^2\right)\le 0.
\endaligned
\]
Despite the nice stability property, it is clear that solving
Problem~\ref{pb:BE} requires expensive numerical strategies in order to deal
with the fully implicit non-linear scheme. For that reason, we considered
other semi-implicit schemes based on the use of the position of the structure
at time $n$ instead of $n+1$. A possible semi-implicit version
of~\eqref{eq:BE} reads
\begin{equation}
\aligned
&\rho_f\left(\frac{\bfu^n-\bfu^n}{\dt},\bfv\right)
+b\left(\bfu^{n},\bfu^{n+1},\bfv\right)+a\left(\bfu^{n+1},\bfv\right)\\
&\hspace{2.0cm}-\left(\div\bfv,p^{n+1}\right)+
c\left(\bflambda^{n+1},\bfv(\bfX^n)\right)=0&&\forall\bfv\in\Huo\\
&\left(\div\bfu^{n+1},q\right)=0&&\forall q\in\Ldo\\
&\dr\left(\frac{\bfX^{n+1}-2\bfX^n+\bfX^{n-1}}{\dt^2},\bfz\right)_\B
+\left(\bbP(\bbF^{n+1}),\Grads\bfz\right)_\B\\
&\hspace{5.1cm}-c\left(\bflambda^{n+1},\bfz\right)=0&&\forall\bfz\in\Hub\\
&c\left(\bfmu,\bfu^{n+1}(\bfX^n)-\frac{\bfX^{n+1}-\bfX^n}{\dt}\right)=0
&&\forall\bfmu\in\bfLambda\\
&\aligned
&\bfu^0=\bfu_0&&\text{in }\Omega\\
&\bfX^0=\bfX_0&&\text{in }\B\\
&\frac{\bfX^0-\bfX^{-1}}{\dt}=\bfX_1&&\text{in }\B.
\endaligned
\endaligned
\label{eq:semi-impl}
\end{equation}
Moreover, in each particular situation, the quantity $\bbP(\bbF^{n+1})$ should
also need a linearization in order to avoid the presence of fully
implicit terms.

The following stability estimate was proved in~\cite{BCG}.

\begin{proposition}

Let us assume that the potential energy density $W$ is a $C^1$ convex
function, then the solution of~\eqref{eq:semi-impl} satisfies
\[
\aligned
&\frac{\rho_f}{2\dt}\left(\|\bfu^{n+1}\|^2_0-\|\bfu^n\|^2_0\right)+
\nu\|\symgrad(\bfu^{n+1})\|^2_0\\
&+\frac{\dr}{2\dt}\left(\left\|\frac{\bfX^{n+1}-\bfX^n}{\dt}\right\|^2_{0,\B}
-\left\|\frac{\bfX^n-\bfX^{n-1}}{\dt}\right\|^2_{0,\B}\right)+
\frac{E(\bfX^{n+1})-E(\bfX^n)}{\dt} \le0.
\endaligned
\]
\label{pr:stab}
\end{proposition}

\begin{remark}

The stability results presented in Proposition~\ref{pr:stab} is a significant
improvement over other schemes used for the approximation of fluid-structure
interactions problems.
A keystone result in this framework is reported in~\cite{causin} where it is
shown that schemes based on the Arbitrary Lagrangian Eulerian (ALE) approach
cannot be stable, when the density of the fluid is close to that of the solid,
unless they are fully implicit. Within the Immersed Boundary Method (IBM),
when finite differences are used for the space discretization, it is shown
that unconditional stability estimates can be obtained in some
circumstances~\cite{nfgk}. In our previous works we have shown a conditional
stability, subject to a CFL condition, for the FE-IBM~\cite{BGHP,CFL}.

\end{remark}

In~\cite{wolf} we investigated how to apply higher order schemes.
We have to pay attention to the term involving the second time derivative of
$\bfX$; as it is common in this case, we reduce the order of the time
derivative by introducing a new variable $\bfw$ corresponding to the first
derivative of $\bfX$ (see also Problem~\ref{pb:DLM-lin}).
For instance, a scheme based on the $\mathsf{BDF2}$ discretization reads
\begin{equation}
\aligned
&\rho_f\left(\frac{3\bfu^{n+1}-4\bfu^n+\bfu^{n-1}}{2\dt},\bfv\right)
+b\left(\bfu^n,\bfu^{n+1},\bfv\right)\\
&\hspace{0.5cm}+a\left(\bfu^{n+1},\bfv\right)
-\left(\div\bfv,p^{n+1}\right)
+c\left(\bflambda^{n+1},\bfv(\bfX^n)\right)=0&&\forall\bfv\in\Huo\\
&\left(\div\bfu^{n+1},q\right)=0&&\forall q\in\Ldo\\
&\left(\frac{3\bfX^{n+1}-4\bfX^n+\bfX^{n-1}}{2\dt},\bfy\right)_\B=
(\bfw^{n+1},\bfy)_\B&&\forall\bfy\in\bfL^2(\B)\\
&\dr\left(\frac{3\bfw^{n+1}-4\bfw^{n}+\bfw^{n-1}}{2\dt},\bfz\right)_\B
+\left(\bbP(\bbF^{n+1}),\Grads\bfz\right)_\B\\
&\hspace{5.7cm}-c\left(\bflambda^{n+1},\bfz\right)=0&&\forall\bfz\in\Hub\\
&c\left(\bfmu,\bfu^{n+1}(\bfX^n)-\frac{3\bfX^{n+1}-4\bfX^n+
\bfX^{n-1}}{2\dt}\right)=0&&\forall\bfmu\in\bfLambda\\
&\aligned
&\bfu^0=\bfu_0&&\text{in }\Omega\\
&\bfX^0=\bfX_0&&\text{in }\B\\
&\bfw^0=\bfX_1&&\text{in }\B.
\endaligned
\endaligned
\label{eq:BDF}
\end{equation}

The following stability estimate was proved in~\cite{wolf}.

\begin{proposition}

Let us assume that the Piola--Kirchhoff tensor is linear
$\bbP(\bbF)=\kappa\bbF$, then the solution of~\eqref{eq:BDF} satisfies
\[
\aligned
&\frac{\rho_f}{4\dt} \left[ \left\|\bfu_h^{n+1}\right\|_\Omega^2 + \left\|2
\bfu_h^{n+1} - \bfu_h^n\right\|_\Omega^2 - \left\|\bfu_h^n\right\|_\Omega^2 -
\left\|2 \bfu_h^n - \bfu_h^{n-1}\right\|_\Omega^2\right.\\
&\qquad \left.  + \left\|\bfu_h^{n+1}-2\bfu_h^n + \bfu_h^{n-1}\right\|_\Omega^2\right]
+ \nu \left\|\symgrad(\bfu_h^{n+1})\right\|_\Omega^2  \\
&\qquad + 
\frac{\dr}{4\dt^2}\left(\|\dot\bfX_h^{n+1}\|_\B^2+\|2\dot\bfX_h^{n+1}-\dot\bfX_h^n\|^2_\B\right.\\
&\qquad\left.
-\|\dot\bfX_h^n\|^2_\B-\|2\dot\bfX_h^n-\dot\bfX_h^{n-1}\|^2_\B
+\|\dot\bfX_h^{n+1}-2\dot\bfX_h^n+\dot\bfX_h^{n-1}\|^2_\B\right) \\
&\qquad +
\frac{\kappa}{4\dt}
\left(\|\bbF_h^{n+1}\|_\B^2+\|2\bbF_h^{n+1}-\bbF_h^n\|^2_\B\right.\\
&\qquad\left.
-\|\bbF_h^n\|^2_\B-\|2\bbF_h^n-\bbF_h^{n-1}\|^2_\B
+\|\bbF_h^{n+1}-2\bbF_h^n+\bbF_h^{n-1}\|^2_\B\right)\le0.
\endaligned
\]

\end{proposition}

We refer the interested reader to~\cite{wolf} for other second order schemes
based on the Crank--Nicolson method and to their corresponding stability
properties which are analogous of the ones presented above. Some numerical
experiments will be presented in Section~\ref{se:numerical}.

\section{Analysis and finite element approximation of the associated saddle
point problem}
\label{se:saddle}

In this section we discuss the finite element discretization in space of our
problem.
We consider the semi-implicit version of one of the schemes introduced in the
previous section. At each time step we have to solve a stationary problem that
we are going to present, approximate, and analyze. We consider
$\Xb\in\bfW^{1,\infty}(\B)^d$ that corresponds to $\bfX^n$ and
$\overline\bfu\in\bfL^\infty(\Omega)$ that corresponds to $\bfu^n$. In this
section we deal with the following Piola--Kirchhoff tensor
\[
\bbP(\bbF)=\kappa\bbF=\kappa\Grads\bfX.
\]
Moreover, we define the following bilinear forms
\[
\aligned
&a_f(\bfu,\bfv)=\alpha(\bfu,\bfv)+a(\bfu,\bfv)+b(\overline\bfu,\bfu,\bfv)
&&\forall\bfu,\bfv\in\Huo\\
&a_s(\bfX,\bfz)=\beta(\bfX,\bfz)_{\B}+\gamma(\Grads\bfX,\Grads\bfz)_{\B}&&
\forall\bfX,\bfz\in\Hub,
\endaligned
\]
where the constants $\alpha$, $\beta$, and $\gamma$ depend on the time step
and on the coefficients of our model. For instance, in the case of backward
Euler method we have
\[
\alpha=\rho_f/\dt,\ \beta=\dr/\dt,\ \gamma=\kappa\dt.
\]
Then, setting $\bfu=\bfu^{n+1}$, $p=p^{n+1}$, $\bfX=\bfX^{n+1}/\dt$,
$\bflambda=\bflambda^{n+1}$, we are led to the following problem.

\begin{problem}[Saddle point problem]
Given $\bff\in\bfL^2(\Omega)$, $\bfg\in\bfL^2(\B)$, and $\bfd\in\bfL^2(\B)$,
find $\bfu\in\Huo$, $p\in\Ldo$, $\bfX\in\Hub$, and $\bflambda\in\bfLambda$
such that
\begin{equation}
\label{eq:saddle}
\aligned
&a_f(\bfu,\bfv)-(\div\bfv,p)+c(\bflambda,\bfv(\Xb))=(\bff,\bfv)
&&\forall\bfv\in\Huo\\
&(\div\bfu,q)=0&&\forall q\in\Ldo\\
&a_s(\bfX,\bfz)-c(\bflambda,\bfz)=(\bfg,\bfz)_{\B}&&\forall\bfz\in\Hub\\
&c(\bfmu,\bfu(\Xb)-\bfX)=c(\bfmu,\bfd)&&\forall\bfmu\in\bfLambda.
\endaligned
\end{equation}
\label{pb:saddle}
\end{problem}

In general, $\bff$, $\bfg$, and $\bfd$ are related to quantities at previous
time steps. For instance, in the case of the backward Euler scheme we have
\[
\bff=\frac{\rho_f}{\dt}\bfu^n,\quad
\bfg=\frac{\dr}{\dt^2}\left(2\bfX^n-\bfX^{n-1}\right),\quad
\bfd=-\frac1{\dt}\bfX^n.
\]

Problem~\ref{pb:saddle}, after converting bilinear forms into linear operators
with natural notation, reads
\[
\left[
\begin{array}{ccc|c}
\m{A}_f&\m{B}_f^\top&\m{0}&\m{C}_f^\top\\
\m{B}_f&\m{0}&\m{0}&\m{0}\\
\m{0}&\m{0}&\m{A}_s&-\m{C}_s^\top\\
\hline
\m{C}_f&\m{0}&-\m{C}_s&\m{0}
\end{array}
\right]
\left[\begin{array}{c}
\m{u}\\\m{p}\\\m{X}\\\hline\m{\lambda}
\end{array}\right]
=\left[\begin{array}{c}
\m{f}\\\m{0}\\\m{g}\\\hline\m{d}
\end{array}\right]
\]
which has a saddle point structure. While the analysis of this problem has
been published in~\cite{BG}, in~\cite{BGarXiv} we observed that it was more
convenient to rearrange the unknowns as follows
\[
\left[
\begin{array}{ccc|c}
\m{A}_f&\m{0}&\m{C}_f^\top&\m{B}_f^\top\\
\m{0}&\m{A}_s&-\m{C}_s^\top&\m{0}\\
\m{C}_f&-\m{C}_s&\m{0}&\m{0}\\
\hline
\m{B}_f&\m{0}&\m{0}&\m{0}
\end{array}
\right]
\left[\begin{array}{c} \m{u}\\ \m{X} \\ \m{\lambda} \\ \hline\m{p} \end{array}\right]
=\left[\begin{array}{c} \m{f}\\ \m{g} \\ \m{d}\\ \hline\m{0}\end{array}\right].
\]

The saddle point structure is evident by introducing the following operators:
$\bbA:\bbV\to\bbV'$ and $\bbB:\bbV\to\Ldo'$ given by
\begin{equation}
\bbA=
\left[
\begin{array}{cc;{2pt/2pt}c}
\m{A}_f&\m{0}&\m{C}_f^\top\\
\m{0}&\m{A}_s&-\m{C}_s^\top\\
\hdashline[2pt/2pt]
\m{C}_f&-\m{C}_s&\m{0}
\end{array}
\right],\qquad
\bbB=
\left[
\begin{array}{cc;{2pt/2pt}c}
\m{B}_f&\m{0}&\m{0}
\end{array}
\right],
\label{eq:underline}
\end{equation}
where $\bbV=\Huo\times\Hub\times\bfLambda$ equipped with the graph norm.
In particular the operator $\bbA$ has itself a saddle point structure which is
highlighted by the dashed lines.  In~\cite{BGarXiv} it is shown that
Problem~\ref{pb:saddle} is well posed by proving the following properties.
\begin{itemize}
\item The operator $\bbA$ is invertible in the kernel of $\bbB$. 
\item The operator $\bbB$ is surjective.
\end{itemize}
Since $\bbA$ is characterized by a saddle point structure, its invertibility
is proved by showing the validity of two inf-sup conditions, while the
surjectivity of $\bbB$ follows from the standard inf-sup condition of
Stokes-like problems. For the sake of completeness, we recall the statements
of the results that are needed in order to prove that $\bbA$ is invertible in
the kernel of $\bbB$.

We start by observing that $\bfV=(\bfv,\bfX,\bfmu)\in\bbV$ belongs to the
kernel of $\bbB$ if and only if $\div\bfv=0$. We recall that the divergence
free subspace of $\Huo$ was denoted by $\Vo$.

In order to study the operator $\bbA$ we use the following kernel (see
also~\eqref{eq:kernel})
\[
\bbK=\left\{(\bfv,\bfz)\in\Vo\times\Hub:c\left(\bfmu,\bfv(\Xb)-\bfz\right)=0\
\forall\bfmu\in\bfLambda\right\}
\]
and we show that there exists $\alpha_0>0$ such that
\[
a_f(\bfu,\bfu)+a_s(\bfX,\bfX)\ge\alpha_0
\left(\|\bfu\|^2_1+\|\bfX\|^2_{1,\B}\right)
\quad\forall(\bfu,\bfX)\in\bbK.
\]
The invertibility of $\bbA$ in the kernel of $\bbB$ is then implied by the
following inf-sup condition: there exists a constant $\beta_0>0$ such that
\[
\sup_{(\bfv,\bfz)\in\Vo\times\Hub}\frac{c\left(\bfmu,\bfv(\Xb)-\bfz\right)}
{\left(\|\bfv\|_1^2+\|\bfz\|^2_{1,\B}\right)^{1/2}}\ge
\beta_0\|\bfmu\|_{\bfLambda}\quad\forall\bfmu\in\bfLambda.
\]

The above estimate holds true for both choices of the bilinear form $c$
defined in~\eqref{eq:choice1} and~\eqref{eq:choice2}, and is a natural
consequence of the definition of the norm of $\bfLambda$.

\subsection{Finite element discretization}

The finite element discretization of Problem~\ref{pb:saddle} is performed by
considering finite dimensional subspaces $\bfV_h\subset\Huo$,
$Q_h\subset\Ldo$, $\bfS_h\subset\Hub$, and $\bfLambda_h\subset\bfLambda$.
We assume that the spaces $\bfV_h$ and $Q_h$ are an inf-sup stable choice for
the approximation of the Stokes problem.

In this paper we consider a more general setting than the one studied
in~\cite{BG} where we assumed that $\bfLambda_h$ and $\bfS_h$ were equal to
each other.

The finite element spaces are constructed starting from three \emph{fixed}
shape-regular meshes: $\T_{\bfV}$ with mesh size $\hx$ for the domain
$\Omega$, $\T_{\bfS}$ with mesh size $\hs$ for the domain $\B$, and
$\T_{\bfLambda}$ with mesh size $\hl$ for the domain $\B$.
The first mesh is associated with the use of the Eulerian variable $\bfx$,
while the other two meshes correspond to the Lagrangian variable $\bfs$.
Here we are assuming that $\Omega$ and $\B$ are polytopes and that $\B$
corresponds to the initial configuration of the solid. If this is not the
case, then further approximations should be introduced. In any case a crucial
property of our model is that the meshes are fixed during the entire evolution
of the system.

The discrete counterpart of Problem~\ref{pb:saddle} can be written as follows.

\begin{problem}[Discrete saddle point problem]
Given $\bff\in\bfL^2(\Omega)$, $\bfg\in\bfL^2(\B)$, and $\bfd\in\bfL^2(\B)$,
find $\bfu_h\in\bfV_h$, $p_h\in Q_h$, $\bfX_h\in\bfS_h$, and
$\bflambda_h\in\bfLambda_h$ such that
\begin{equation}
\label{eq:saddleh}
\aligned
&a_f(\bfu_h,\bfv)-(\div\bfv,p_h)+\hat c(\bflambda_h,\bfv(\Xb))=(\bff,\bfv)
&&\forall\bfv\in\bfV_h\\
&(\div\bfu_h,q)=0&&\forall q\in Q_h\\
&a_s(\bfX_h,\bfz)-\hat c(\bflambda_h,\bfz)=(\bfg,\bfz)_{\B}&&\forall\bfz\in\bfS_h\\
&\hat c(\bfmu,\bfu_h(\Xb)-\bfX_h)=\hat c(\bfmu,\bfd)&&\forall\bfmu\in\bfLambda_h.
\endaligned
\end{equation}
\label{pb:saddleh}
\end{problem}

In the formulation presented above we used the notation $\hat c$ for the
discrete realization of the bilinear form $c$ considered in
Problem~\ref{pb:saddle}. Let us detail how this realization looks like in the
two cases described in~\eqref{eq:choice1} and~\eqref{eq:choice2}.
If $\bfLambda=\bfLambda_1$, observing that any reasonable finite element space
$\bfLambda_h$ is included in $\bfL^2(\B)$, it is possible to identify the
duality pairing $c_1$ with the inner product of $\bfL^2(\B)$, so that we take
\[
\hat c_1(\bfmu,\bfz)=(\bfmu,\bfz)_{\B}.
\]
On the other hand, in the case $\bfLambda=\bfLambda_2$ we can take the same
bilinear form as in the continuous case
\[
\hat c_2(\bfmu,\bfz)=(\bfmu,\bfz)_{\B}+(\Grads\bfmu,\Grads\bfz)_{\B}.
\]
We are going to use the same notation $c$ for both approaches as for the
continuous case. When we need to refer explicitly to one of the two
formulations, we shall use the full notation.

The analysis of the discrete problem makes use of the same technique that we
described above for the continuous case. We report the main ingredients of the
proof in a more general setting than it was presented in~\cite{BGarXiv}; this
is also the occasion to amend some detail of~\cite{BG}.

Using the notation introduced above for the space $\bbV$, we introduce as
follows the bilinear forms $\calA:\bbV\times\bbV\to\RE$ and
$\calB:\bbV\times\Ldo\to\RE$ in order to highlight the saddle point structure
of the problem and to make easier the description of the result
\[
\aligned
&\calA(\bfU,\bfV)=a_f(\bfu,\bfv)+a_s(\bfX,\bfz)+c(\bflambda,\bfv(\Xb)-\bfz)
-c(\bfmu,\bfu(\Xb)-\bfX)\\
&\calB(\bfV,q)=(\div\bfv,q),
\endaligned
\]
where we used the notation $\bfU=(\bfu,\bfX,\bflambda)$ and
$\bfV=(\bfv,\bfz,\bfmu)$.
It is clear that the bilinear forms $\calA$ and $\calB$ correspond to the
operators $\bbA$ and $\bbB$ defined above.

We denote by $\bbV_h=\bfV_h\times\bfS_h\times\bfLambda_h$ the subspace of
$\bbV$ that we are using for the approximation. Hence,
Problem~\ref{pb:saddleh} reads: given $\bff\in\bfL^2(\Omega)$,
$\bfg\in\bfL^2(\B)$, and $\bfd\in\Hub$, find $(\bfU_h,p_h)\in\bbV_h\times Q_h$
such that
\begin{equation}
\label{eq:misto}
\aligned
&\calA(\bfU_h,\bfV)+\calB(\bfV,p_h)=(\bff,\bfv)+(\bfg,\bfz)_{\B}-c(\bfmu,\bfd)
&&\forall\bfV\in\bbV_h\\
&\calB(\bfU_h,q)=0&&\forall q\in Q_h.
\endaligned
\end{equation}

Let $\bfV_{0,h}$ be the subset of $\bfV_h$ containing the discretely
divergence free vectorfields, that is $\bfv_h\in\bfV_{0,h}$ if and only if
\[
(\div\bfv_h,q)=0\quad\forall q\in Q_h.
\]
We are going to use the discrete kernel
\[
\bbK_h=\left\{(\bfv_h,\bfz_h)\in\bfV_{0,h}\times\bfS_h:
c\left(\bfmu,\bfv_h(\Xb)-\bfz_h\right)=0\ \forall\bfmu\in\bfLambda_h\right\}.
\]

We state the following compatibility between the spaces $\bfS_h$ and
$\bfLambda_h$ that will be useful in the sequel.

\begin{ass}
There exists a constant $\zeta>0$ such that for all $\bfmu_h\in\bfLambda_h$ it
holds
\begin{equation}
\sup_{\bfz_h\in\bfS_h}
\frac{c(\bfmu_h,\bfz_h)}{\|\bfz_h\|_{1,\B}}\ge\zeta\|\bfmu_h\|_{\bfLambda}.
\label{eq:ass}
\end{equation}
\label{as:ass}
\end{ass}

In order to show the stability of~\eqref{eq:misto} we need to prove the
following inf-sup conditions~\cite{bbf}.

\begin{itemize}
\item There exists $\gamma_1>0$ such that
\begin{equation}
\inf_{\bfU_h\in\bbK_h}\sup_{\bfV_h\in\bbK_h}
\frac{\calA(\bfU_h,\bfV_h)}{\|\bfU_h\|_{\bbV}\|\bfV_h\|_{\bbV}}\ge\gamma_1.
\label{eq:elker}
\end{equation}
\item There exists $\gamma_2>0$ such that
\begin{equation}
\inf_{q_h\in Q_h}\sup_{\bfV_h\in\bbV_h}
\frac{\calB(\bfV_h,q_h)}{\|q_h\|_0\|\bfV_h\|_{\bbV}}\ge\gamma_2.
\label{eq:infsup}
\end{equation}
\end{itemize}

The inf-sup condition for the bilinear form $\calB$ is immediate if the spaces
$\bfV_h$ and $Q_h$ are a good Stokes pair. Indeed it is easy to see that
\[
\inf_{q_h\in Q_h}\sup_{\bfV_h\in\bbV_h}
\frac{\calB(\bfV_h,q_h)}{\|q_h\|_0\|\bfV_h\|_{\bbV}}=
\inf_{q_h\in Q_h}\sup_{\bfv_h\in\bfV_h}
\frac{(\div\bfv_h,q_h)}{\|q_h\|_0\|\bfv_h\|_1}\ge\gamma_2,
\]
where $\gamma_2$ is the inf-sup constant related to $\bfV_h$ and $Q_h$ for the
divergence operator.

In order to show the inf-sup condition for the bilinear form $\calA$, we
start with the following proposition.

\begin{proposition}
For all $\beta\ge0$, there exists a constant $\alpha_1>0$ not depending on the
mesh sizes such that
\[
a_f(\bfu_h,\bfu_h)+a_s(\bfX_h,\bfX_h)\ge\alpha_1
\left(\|\bfu_h\|^2_1+\|\bfX_h\|^2_{1,\B}\right)
\quad\forall(\bfu_h,\bfX_h)\in\bbK_h.
\]
\label{pr:elker}
\end{proposition}

\begin{proof}
This proposition extends the conclusions of~\cite[Prop.~7]{BGarXiv}. For
$\beta>0$, the result follows directly from
\[
\aligned
a_f(\bfu_h,\bfu_h)+a_s(\bfX_h,\bfX_h)&\ge C\|\bfu_h\|^2_1+
\beta\|\bfX_h\|^2_{0,\B}+\kappa\|\Grads\bfX_h\|^2_{0,\B}\\
&\ge C\|\bfu_h\|^2_1+\min(\beta,\kappa)\|\bfX_h\|^2_{1,\B}.
\endaligned
\]

For $\beta=0$, we have
\begin{equation}
a_f(\bfu_h,\bfu_h)+a_s(\bfX_h,\bfX_h)
\ge C\|\bfu_h\|^2_1+\kappa\|\Grads\bfX_h\|^2_{0,\B}.
\label{eq:desired}
\end{equation}

The next step is to show that we can control $\|\bfX_h\|_{\OB}$ by the right
hand side of~\eqref{eq:desired}. This can be done at once for both possible
choices of $\bfLambda$ and $c$.
In order to use the Poincar\'e inequality we split $\bfX_h$ as the sum of its
mean value $\Xoh$ and the rest, so that
\[
\|\bfX_h\|_{\OB}\le\|\Xoh\|_{\OB}+\|\bfX_h-\Xoh\|_{\OB}\le
\|\Xoh\|_{\OB}+C\|\Grads\bfX_h\|_{\OB}.
\]

The constant part $\Xoh$ can be estimated by using the fact that the finite
element space $\bfLambda_h$ contains the global constant functions as follows.
Since $(\bfu_h,\bfX_h)\in\bbK_h$ we have
\[
c(\bfmu_h,\Xoh)=c(\bfmu_h,\bfu_h(\Xb))-c(\bfmu_h,\bfX_h-\Xoh)\qquad
\forall\bfmu_h\in\bfLambda_h.
\]
Choosing $\bfmu_h=\Xoh$ we obtain
\[
\|\Xoh\|_{\OB}^2=c(\Xoh,\bfu_h(\Xb))\le\|\Xoh\|_{\OB}\|\bfu_h(\Xb)\|_{\OB}.
\]
Indeed, if $\bfmu_h$ is constant then the term $c(\bfmu_h,\bfX_h-\Xoh)$
vanishes and, even in the case when $c$ is the scalar product in $\Hub$,
the term involving $\Grads\bfmu_h$ vanishes so that $c$ acts as the scalar
product in $\bfL^2(\B)$.
Hence, we get the final bound $\|\Xoh\|_{\OB}\le\|\bfu_h\|_0$.
\end{proof}

The next step consists in showing the following uniform inf-sup condition.

\begin{proposition}

Let us suppose that Assumption~\ref{as:ass} is satisfied. Then, for
$\beta_1=\zeta$ from~\eqref{eq:ass} we have
\[
\sup_{(\bfv_h,\bfz_h)\in\bfV_{0,h}\times\bfS_h}
\frac{c\left(\bfmu_h,\bfv_h(\Xb)-\bfz_h\right)}
{\left(\|\bfv_h\|_1^2+\|\bfz_h\|^2_{1,\B}\right)^{1/2}}\ge
\beta_1\|\bfmu_h\|_{\bfLambda_h}\quad\forall\bfmu_h\in\bfLambda_h.
\]
\label{pr:infsupA2}
\end{proposition}

\begin{proof}
Using Assumption~\ref{as:ass} we have
\[
\zeta\|\bfmu_h\|_{\bfLambda}
\le\sup_{\bfz_h\in\bfS_h}\frac{c(\bfmu_h,\bfz_h)}{\|\bfz_h\|_{1,\B}}
\le\sup_{(\bfv_h,\bfz_h)\in\bfV_{0,h}\times\bfS_h}
\frac{c(\bfmu_h,\bfv_h(\Xb)-\bfz_h)}
{(\|\bfv_h\|^2_1+\|\bfz_h\|^2_{1,\B})^{1/2}}.
\]
\end{proof}

We now present some possible choices of $\bfS_h$ and $\bfLambda_h$ for which
Assumption~\ref{as:ass} holds true.

We start by considering the case when $\bfLambda=\bfLambda_2$ and the bilinear
form $\hat c=c_2$, namely it corresponds to the scalar product in $\Hub$.
The most natural situation is when $\bfLambda_h\subseteq\bfS_h$ which is the
object of the following proposition. This condition is satisfied, for
instance, if the mesh $\T_{\bfS}$ is the same as $\T_{\bfLambda}$ or a
refinement of it and the space $\bfS_h$ contains polynomials of degree higher
than or equal to those in $\bfLambda_h$.

\begin{proposition}

Let $\bfLambda=\bfLambda_2$ and the bilinear form $\hat c=c_2$ be the scalar
product in $\Hub$. If $\bfLambda_h\subseteq\bfS_h$ then the inf-sup
condition~\eqref{eq:ass} is satisfied.
\label{pr:facile}
\end{proposition}

\begin{proof}

Given $\bfmu_h\in\bfLambda_h$, since $\bfLambda_h\subseteq\bfS_h$, it is
possible to take $\bfz_h=\bfmu_h$ so that
\[
\|\bfmu_h\|_{\bfLambda}=
\frac{(\bfmu_h,\bfz_h)_{\B}+(\Grads\bfmu_h,\Grads\bfz_h)_{\B}}
{\|\bfz_h\|_{1,\B}}=\frac{c_2(\bfmu_h,\bfz_h)}{\|\bfz_h\|_{1,\B}}
\le\sup_{\bfz_h\in\bfS_h}\frac{c_2(\bfmu_h,\bfz_h)}{\|\bfz_h\|_{1,\B}}.
\]
Hence the inf-sup condition~\eqref{eq:ass} holds true with $\zeta=1$.

\end{proof}

Let us now consider the case when $\bfLambda=\bfLambda_1$ is the dual of
$\Hub$ and the bilinear form $\hat c=\hat c_1$ is the scalar product in
$\bfL^2(\B)$.
We take again the most natural situation when $\bfLambda_h\subseteq\bfS_h$ as
in Proposition~\ref{pr:facile}. In this case, however, the validity of the
inf-sup condition~\eqref{eq:ass} relies on an additional hypothesis that
involves the $\Hub$-stability of the $\bfL^2(\B)$-projection onto $\bfS_h$.

\begin{proposition}

Let $\bfLambda=\bfLambda_1=(\Hub)'$ and the bilinear form $\hat c=\hat c_1$ be
the scalar product in $\bfL^2(\B)$.
Let $P_0$ denote the $\bfL^2(\B)$-projection from $\Hub$ onto $\bfS_h$ and
assume that there is a constant $C$ such that
\begin{equation}
\|P_0\bfz\|_{1,\B}\le C_0\|\bfz||_{1,\B}\quad\forall\bfz\in\Hub.
\label{eq:proj}
\end{equation}
Then, if $\bfLambda_h\subseteq\bfS_h$, the inf-sup condition~\eqref{eq:ass} is
satisfied.
\label{pr:proj}
\end{proposition}

\begin{proof}

By definition of the norm in $\bfLambda$ there exists $\tilde\bfz\in\Hub$ such
that
\[
\|\bfmu_h\|_{\bfLambda}=\frac{\hat c_1(\bfmu_h,\tilde\bfz)}{\|\tilde\bfz\|_{1,\B}}
=\frac{\hat c_1(\bfmu_h,P_0\tilde\bfz)}{\|\tilde\bfz\|_{1,\B}},
\]
where in the last equality we used $\bfLambda_h\subseteq\bfS_h$. Finally,
using the $\Hub$-stability of $P_0$ stated in~\eqref{eq:proj}, we get
\[
\|\bfmu_h\|_{\bfLambda}\le C_0
\frac{\hat c_1(\bfmu_h,P_0\tilde\bfz)}{\|P_0\tilde\bfz\|_{1,\B}}\le C_0
\sup_{\bfz_h\in\bfS_h}\frac{\hat c_1(\bfmu_h,\bfz_h)}{\|\bfz_h\|_{1,\B}}.
\]
Hence, the proposition is proved with $\zeta=1/C_0$.

\end{proof}

The cases considered in Propositions~\ref{pr:facile} and~\ref{pr:proj}
generalize the situation discussed in~\cite{BG}, where $\bfLambda_h$ was
chosen equal to $\bfS_h$. It will be the object of further investigation to
explore other possible combinations for $\bfLambda_h$ and $\bfS_h$. In
particular, it would be quite natural to take a space of discontinuous finite
elements for the multiplier in the case when $\bfLambda=(\Hub)'$. On the other
hand, our present analysis does not cover for instance the situation when
$\bfLambda_h$ is the space of piecewise constants and $\bfS_h$ is the space of
continuous piecewise linear elements in each component:
Assumption~\ref{as:ass} requires $\dim(\bfS_h)\ge\dim(\bfLambda_h)$ as a
necessary condition, which is not satisfied on general meshes for this choice
of finite elements.

The results of this section can be summarized in the following stability and
convergence theorems.

\begin{thm}

Under the assumptions of Propositions~\ref{pr:elker} and \ref{pr:infsupA2},
there exists $\gamma_1>0$ such that the inf-sup condition~\eqref{eq:elker} is
satisfied.

If, moreover, $\bfV_h$ and $Q_h$ satisfy the usual compatibility condition for
the solution of the Stokes problem, then the inf-sup
condition~\eqref{eq:infsup} holds true.

\end{thm}

\begin{proof}

The results of this theorem follow from the previous propositions with
classical arguments related to the stability of saddle point
problems~\cite{bbf} (see also~\cite{XZ}).

The inf-sup condition~\eqref{eq:elker} is the necessary and sufficient
condition for the uniform invertibility of the matrix
\[
\underline{\bbA}=\left[
\begin{array}{cc;{2pt/2pt}c}
\underline{\m{A}}_f&\m{0}&\underline{\m{C}}_f^\top\\
\m{0}&\underline{\m{A}}_s&-\underline{\m{C}}_s^\top\\
\hdashline[2pt/2pt]
\underline{\m{C}}_f&-\underline{\m{C}}_s&\m{0}
\end{array}
\right]
\]
restricted to the discrete kernel of the matrix
\[
\underline{\bbB}=
\left[
\begin{array}{cc;{2pt/2pt}c}
\underline{\m{B}}_f&\m{0}&\m{0}
\end{array}
\right],
\]
where the blocks $\underline{\m{A}}_f$, $\underline{\m{A}}_s$,
$\underline{\m{B}}_f$, $\underline{\m{C}}_f$, and $\underline{\m{C}}_s$ are
matrix representations of the corresponding operators in~\eqref{eq:underline}.

Proposition~\ref{pr:elker} states the uniform invertibility of the block
\[
\left[
\begin{array}{cc}
\underline{\m{A}}_f&\m{0}\\
\m{0}&\underline{\m{A}}_s
\end{array}
\right]
\]
restricted to the kernel $\bbK_h$ of
\[
\left[
\begin{array}{cc}
\underline{\m{C}}_f&-\underline{\m{C}}_s
\end{array}
\right].
\]

Proposition~\ref{pr:infsupA2} states the surjectivity of this last matrix with
uniform bound of its inverse.

Putting things together, we get the inf-sup condition~\eqref{eq:elker}.
The second part of the theorem has been discussed after
formula~\eqref{eq:infsup}.

\end{proof}

From the stability of the discrete problems, the convergence result follows in
a straightforward way.

\begin{thm}

Let $\bfV_h$ and $Q_h$ satisfy the usual compatibility condition for the
solution of the Stokes problem and let us assume the hypotheses of
Propositions~\ref{pr:elker} and \ref{pr:infsupA2}. Then there exists a unique
solution $(\bfu_h,p_h,\bfX_h,\bflambda_h)$ to Problem~\ref{pb:saddleh}. Let
$(\bfu,p,\bfX,\bflambda)$ be the solution to the continuous
Problem~\ref{pb:saddle}. Then the following optimal error estimate holds true
\[
\aligned
&\|\bfu-\bfu_h\|_1+\|p-p_h\|_0+\|\bfX-\bfX_h\|_{1,\B}
+\|\bflambda-\bflambda_h\|_{\bfLambda}\\
&\qquad\le C\left( \inf_{\bfv\in\bfV_h}\|\bfu-\bfv\|_1+
\inf_{q\in Q_h}\|p-q\|_0+
\inf_{\bfz\in\bfS_h}\|\bfX-\bfz\|_{1,\B}+
\inf_{\bfmu\in\bfLambda_h}\|\bflambda-\bfmu\|_{\bfLambda}\right).
\endaligned
\]

\end{thm}

\section{Numerical results}
\label{se:numerical}

In this section we collect some numerical experiments that have been reported
in previous papers and that confirm the effectiveness of the method.

We start with a test reported in~\cite{BCG} confirming the unconditional
stability stated in Proposition~\ref{pr:stab}. We consider a benchmark test
problem where at the initial time the solid occupies an ellipsoidal region
which evolves approaching a circular equilibrium configuration. We approximate
the problem by using the enhanced Bercovier--Pironneau element introduced and
analyzed in~\cite{budini}, consisting in a P1-iso-P2 discretization of the
velocities and in a continous P1 discretization of the pressures augmented by
piecewise constant functions in order to improve the mass conservation of the
scheme.
We compare our fictitious domain approach FE-DLM (solid line) with the FE-IBM
scheme (dashed line), see~\cite{CFL}.
We take $\Omega$ equal to the square of side $(-1,1)$ and we study a
ring-shaped immersed structure with reference configuration given by
$\B=\{\bfx\in\RE^2:0.3\le|\bfx|\le0.5\}$. For symmetry reasons, we reduce the
computation to a quarter of $\Omega$ so that the configuration is the one
reported schematically in Figure~\ref{fg:mesh}.

\begin{figure}
        \centering
        \includegraphics[height=0.4\textwidth]{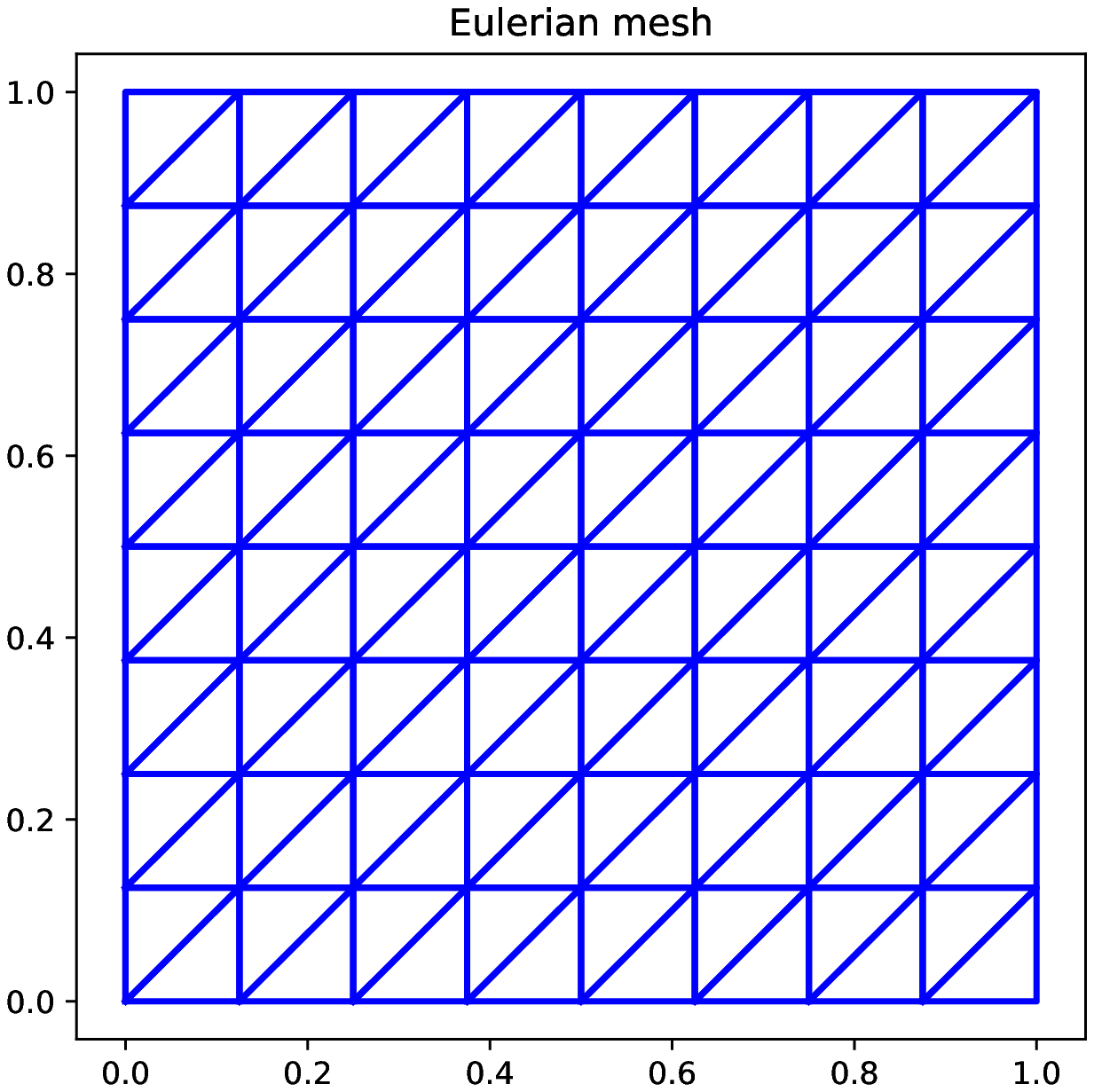}
        \includegraphics[height=0.4\textwidth]{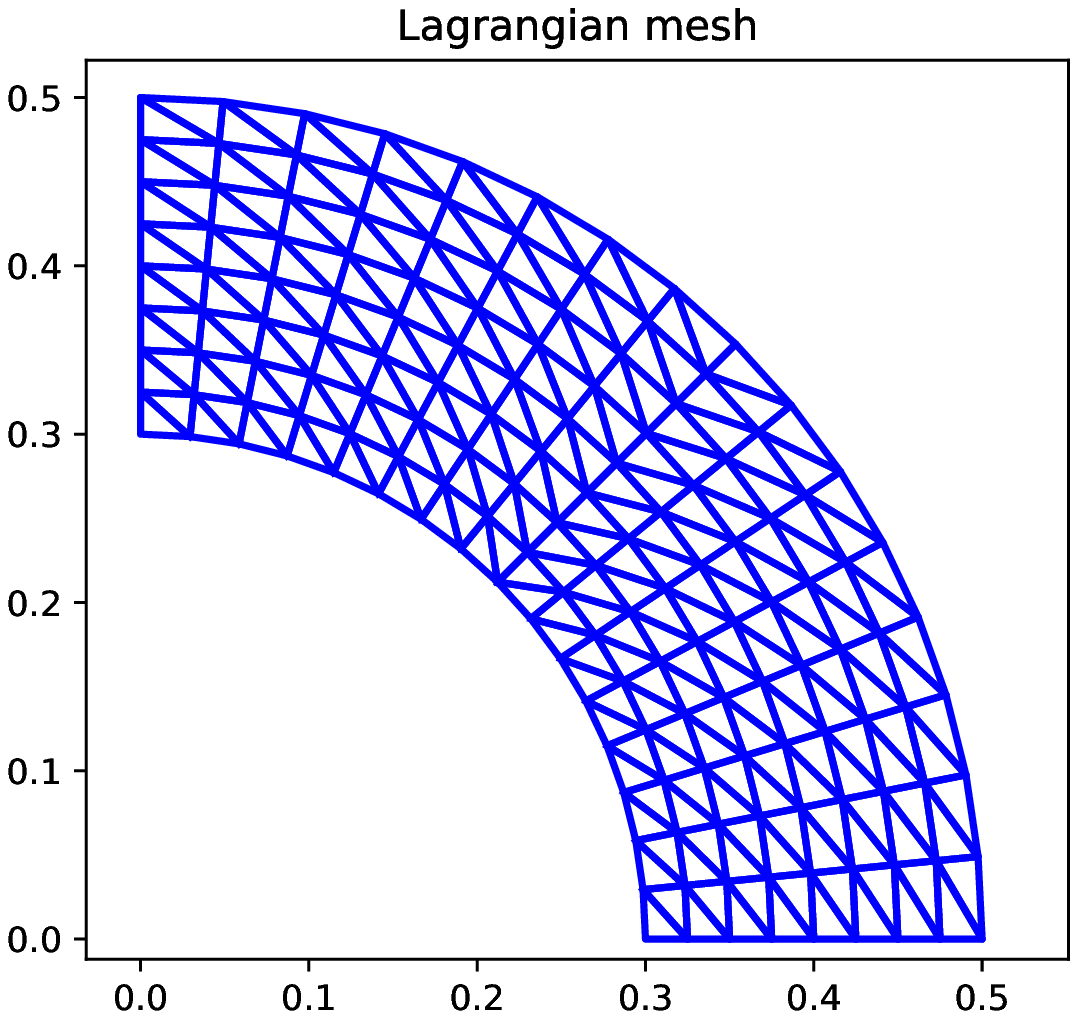}
        \caption{Sketch of the meshes used for the fluid and the structure}
        \label{fg:mesh}
\end{figure}

The materials properties are $\rho_f=1$, $\nu=0.05$, $\dr=0.3$, and
$\kappa=1$.
The solid mesh size is equal to $1/8$ and the figures show the behavior of the
following energy ratio as a function of the time step and of the fluid mesh
size
\begin{equation}
\label{eq:ener_comp}
\Pi(\bfX_h^{n},\bfu_h^{n})=\frac{\rho_f}{2}\|\bfu_h^{n}\|^2_0
+\frac{\dr}{2}\left\|\frac{\bfX_h^{n}-\bfX_h^{n-1}}{\dt}\right\|^2_{0,\B}
+
E(\bfX_h^{n}).
\end{equation}
\begin{figure}
 \centering
 \subcaptionbox{\tiny{$\Delta t = 10^{-1}$, $h_x = 1/4$}.}
   {\includegraphics[scale=.2]{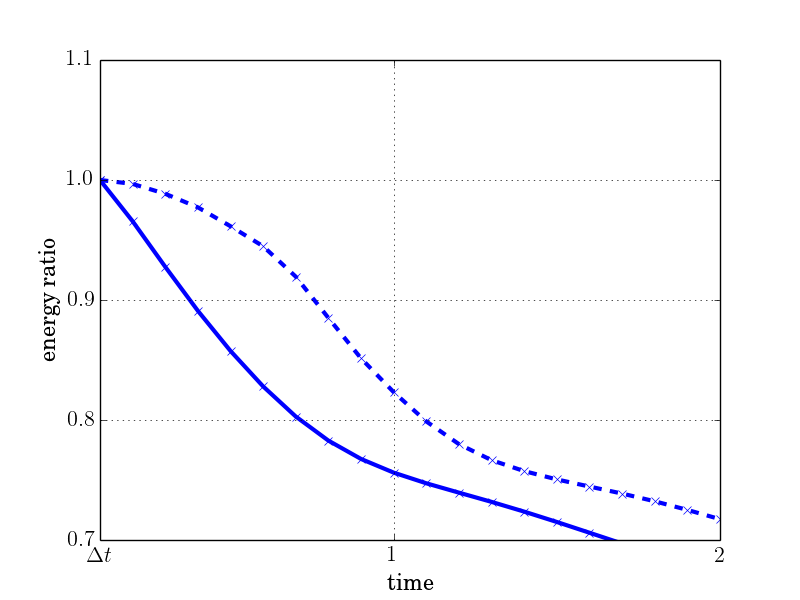}}
 \subcaptionbox{\tiny{$\Delta t = 10^{-1}$, $h_x = 1/8$}.}
   {\includegraphics[scale=.2]{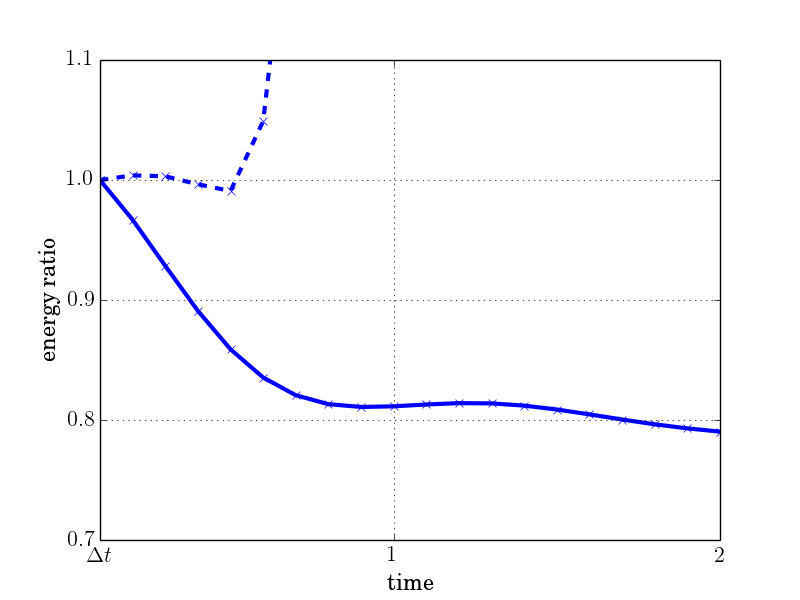}}
 \subcaptionbox{\tiny{$\Delta t = 10^{-1}$, $h_x = 1/16$}.}
   {\includegraphics[scale=.2]{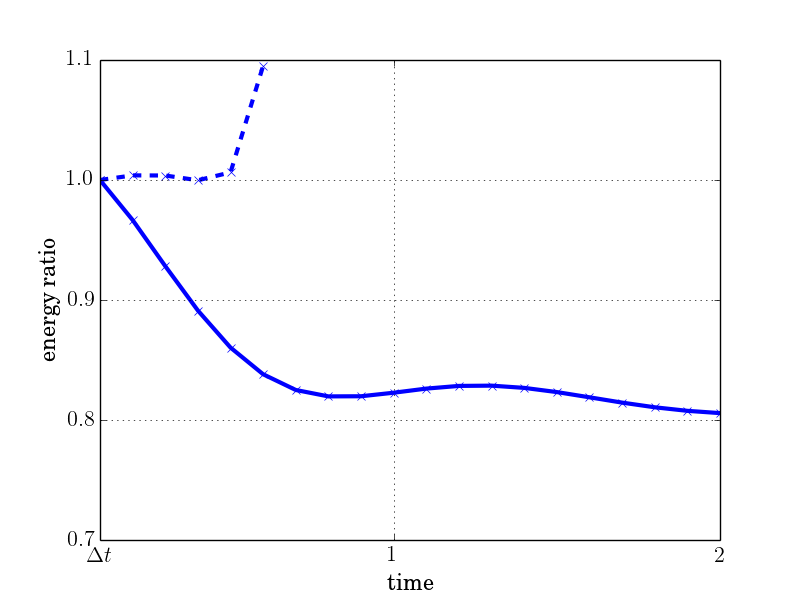}}\\
  \subcaptionbox{\tiny{$\Delta t = 5\cdot10^{-2}$, $h_x = 1/4$}.}
   {\includegraphics[scale=.2]{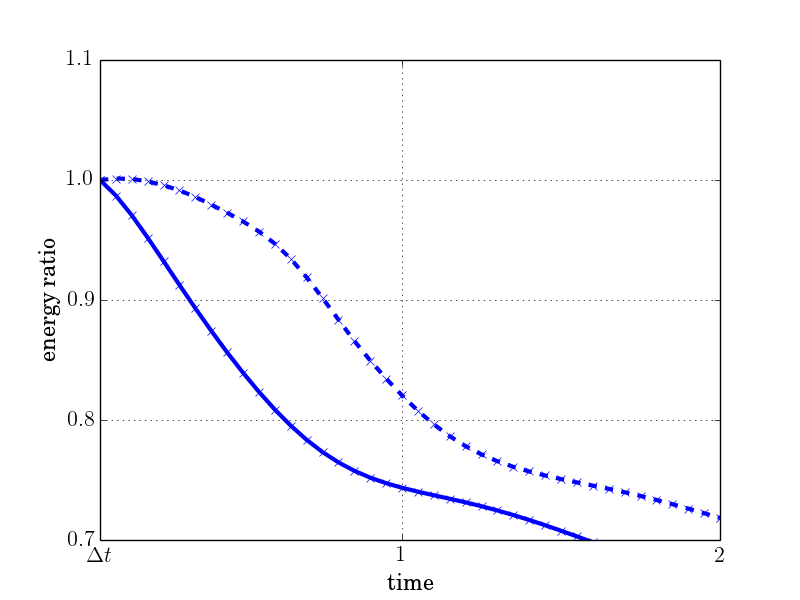}}
 \subcaptionbox{\tiny{$\Delta t = 5\cdot10^{-2}$, $h_x = 1/8$}.}
   {\includegraphics[scale=.2]{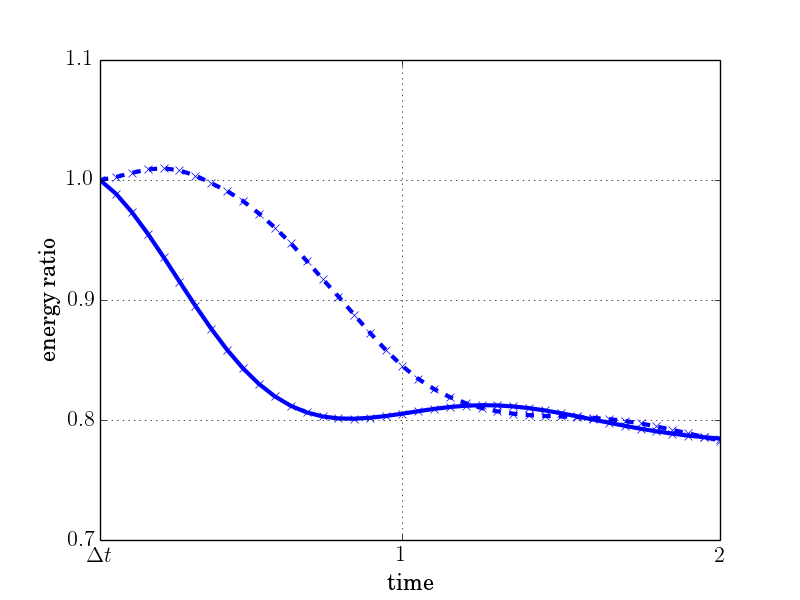}}
 \subcaptionbox{\tiny{$\Delta t = 5\cdot10^{-2}$, $h_x = 1/16$}.}
   {\includegraphics[scale=.2]{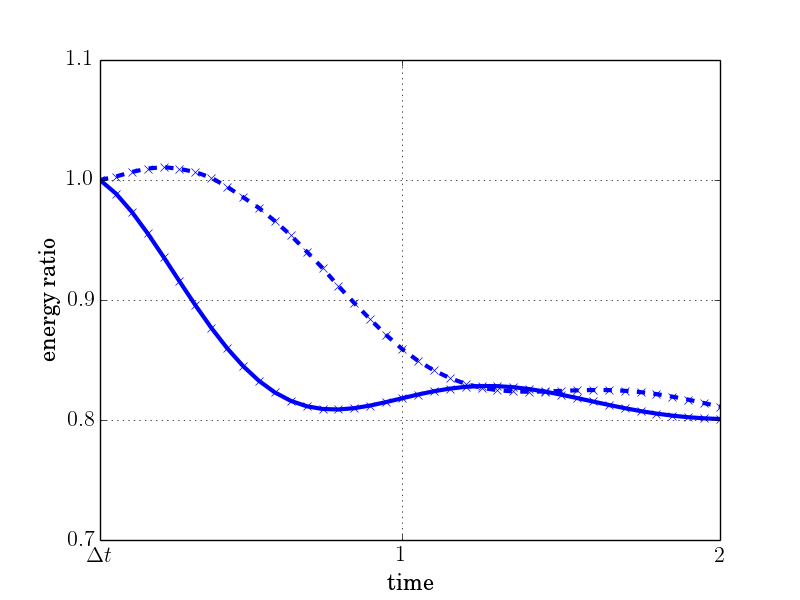}}
 \caption{Evolution of the quantity
$\Pi(\bfX_h^{n},\bfu_h^{n})/\Pi(\bfX_h^{0},\bfu_h^{0})$ (see
Equation~\eqref{eq:ener_comp}) for different $\dt$ when $h_x$ varies.
The solid line corresponds to the formulation FE-DLM described in this paper,
while the dashed line refers to the FE-IBM scheme which is only conditionally
stable}
\label{fig:thick_energy_rho03}
\end{figure}

In Table~\ref{TableConvNavierSemiimpl16} we report the results presented
in~\cite{wolf} about the convergence rates in time when different time schemes
are used.
In these computations the mesh of $\Omega$ is based on a subdivision of
$(-1,1)$ in $32$ equal subintervals and the structure is modeled by a
Lagrangian mesh obtained by halving the meshsize of the one reported in
Figure~\ref{fg:mesh}.
The fluid is initially at rest and the structure is stretched by a factor
$1.4$ in the vertical direction and shrunk by the same factor in the
horizontal direction.

The physical parameters are $\rho_f=\rho_s=1$, $\nu=0.1$, $\kappa=10$, and
$T=1$.

We consider \BD (semi-implicit backward Euler~\eqref{eq:semi-impl}), \BDF
(see~\eqref{eq:BDF}), and two variants of Crank--Nicolson scheme. We denote by
\CNM the case when the nonlinear terms are evaluated using the midpoint rule
and by \CNT the case when the trapezoidal rule is used.

The reference solution is calculated by using a smaller timestep with the
\BDF scheme.

\begin{table}
        \centering
        \begin{tabular}{| l | l r | l r | l r | l r |}
                \multicolumn{9}{c}{Fluid velocity} \\
                \hline
                & \multicolumn{2}{|c|}{\BD} & \multicolumn{2}{|c|}{\BDF} & \multicolumn{2}{|c|}{\CNM} & \multicolumn{2}{|c|}{\CNT}\\
                $\dt$ & $L^2$ error & rate & $L^2$ error & rate & $L^2$ error & rate & $L^2$ error & rate\\
                \hline  
                $0.05$   &$9.18\cdot 10^{-2}$&      &$3.89\cdot 10^{-2}$&      &$2.36\cdot 10^{-1}$&      &$2.39\cdot 10^{-1}$& \\
                $0.025$  &$5.05\cdot 10^{-2}$&$0.86$&$8.59\cdot 10^{-3}$&$2.18$&$7.54\cdot 10^{-2}$&$1.64$&$7.06\cdot 10^{-2}$&$1.76$\\
                $0.0125$ &$2.63\cdot 10^{-2}$&$0.94$&$3.32\cdot 10^{-3}$&$1.37$&$4.24\cdot 10^{-2}$&$0.83$&$2.22\cdot 10^{-2}$&$1.67$\\
                $0.00625$&$1.33\cdot 10^{-2}$&$0.98$&$1.40\cdot 10^{-3}$&$1.24$&$2.19\cdot 10^{-2}$&$0.96$&$4.19\cdot 10^{-3}$&$2.40$\\
                \hline
        \end{tabular}

        \begin{tabular}{| l | l r | l r | l r | l r |}
                \multicolumn{9}{c}{Structure deformation} \\
                \hline
                & \multicolumn{2}{|c|}{\BD} & \multicolumn{2}{|c|}{\BDF} & \multicolumn{2}{|c|}{\CNM} & \multicolumn{2}{|c|}{\CNT}\\
                $\dt$ & $L^2$ error & rate & $L^2$ error & rate & $L^2$ error & rate & $L^2$ error & rate\\
                \hline
                $0.05$   &$2.03\cdot 10^{-3}$&      &$7.86\cdot 10^{-4}$&      &$1.81\cdot 10^{-3}$&      &$6.51\cdot 10^{-4}$& \\
                $0.025$  &$1.06\cdot 10^{-3}$&$0.93$&$3.28\cdot 10^{-4}$&$1.26$&$9.75\cdot 10^{-4}$&$0.89$&$1.31\cdot 10^{-4}$&$2.31$\\
                $0.0125$ &$5.34\cdot 10^{-4}$&$1.00$&$1.44\cdot 10^{-4}$&$1.18$&$5.10\cdot 10^{-4}$&$0.93$&$4.82\cdot 10^{-5}$&$1.44$\\
                $0.00625$&$2.69\cdot 10^{-4}$&$0.99$&$6.31\cdot 10^{-5}$&$1.19$&$2.55\cdot 10^{-4}$&$1.00$&$1.29\cdot 10^{-5}$&$1.90$\\
                \hline
        \end{tabular}
        \caption{Convergence results for the semi-implicit scheme on the fine mesh}
        \label{TableConvNavierSemiimpl16}
\end{table}

We conclude this section by showing the evolution of the structure
corresponding to the last example, see Figure~\ref{fg:annulus}. A similar
example corresponding to a square structure was reported in~\cite{ecmi}, see
Figure~\ref{fg:rectangulus}

\begin{figure}
        \centering
        \includegraphics[width=0.4\textwidth]{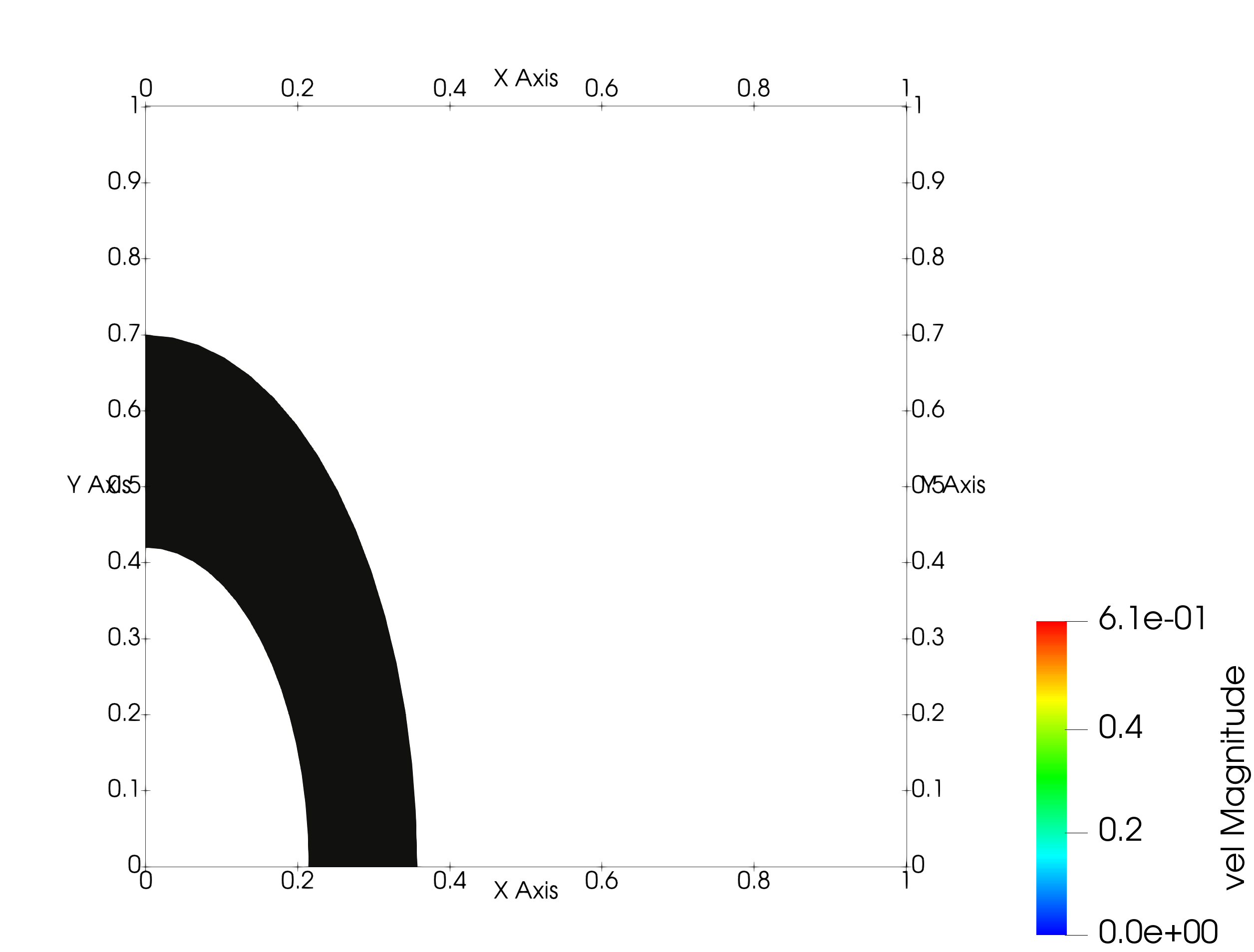}
        \includegraphics[width=0.4\textwidth]{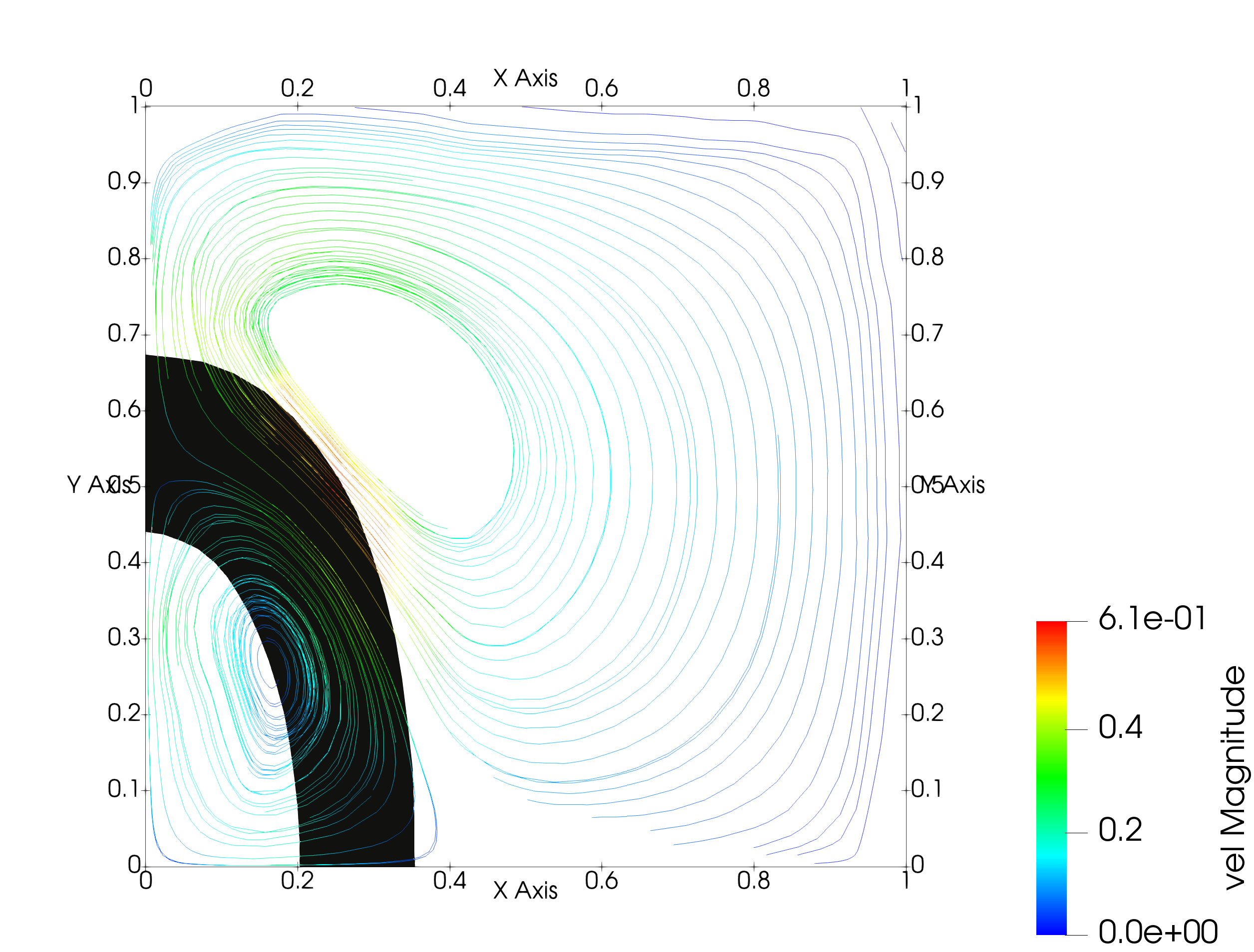}
        \includegraphics[width=0.4\textwidth]{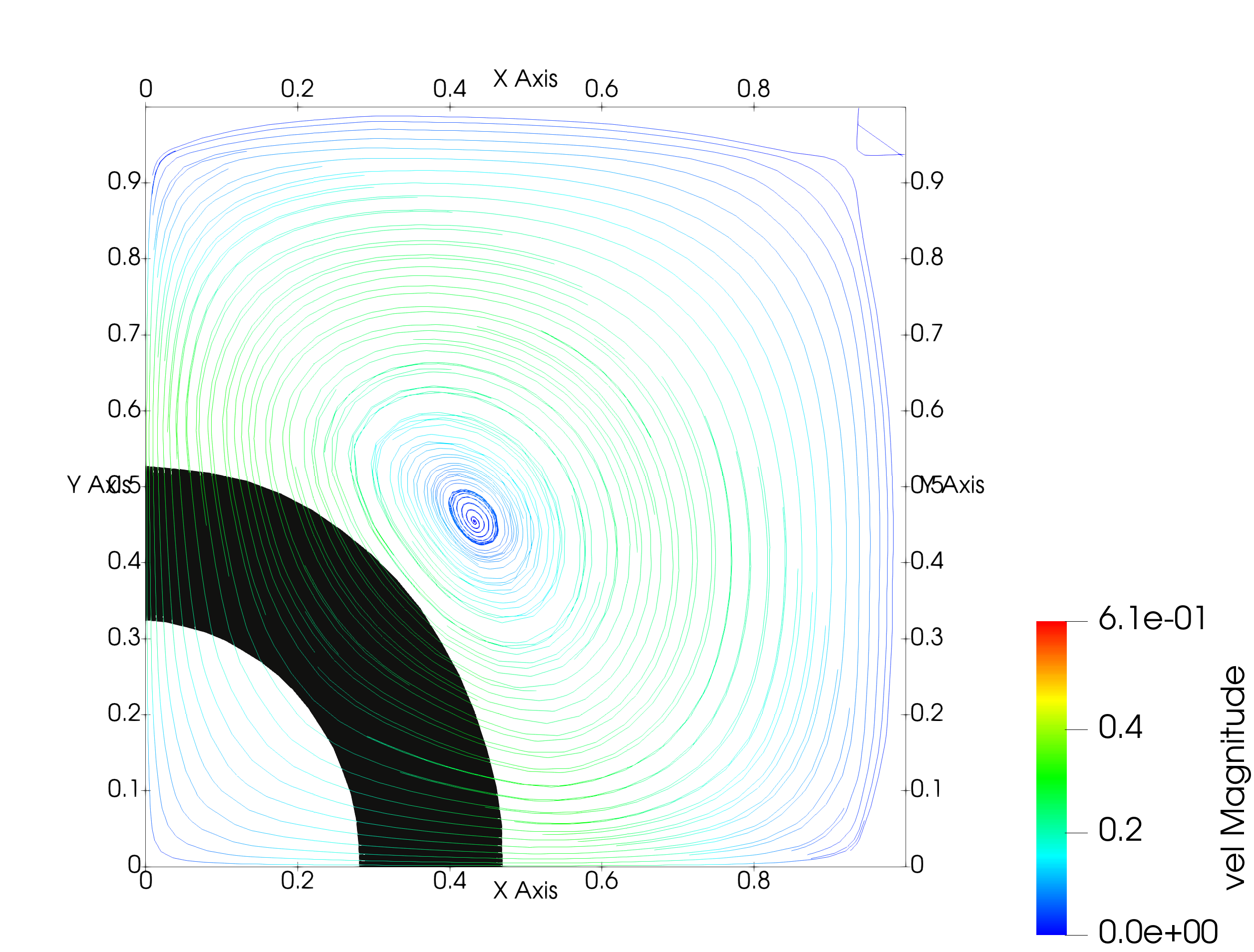}
        \includegraphics[width=0.4\textwidth]{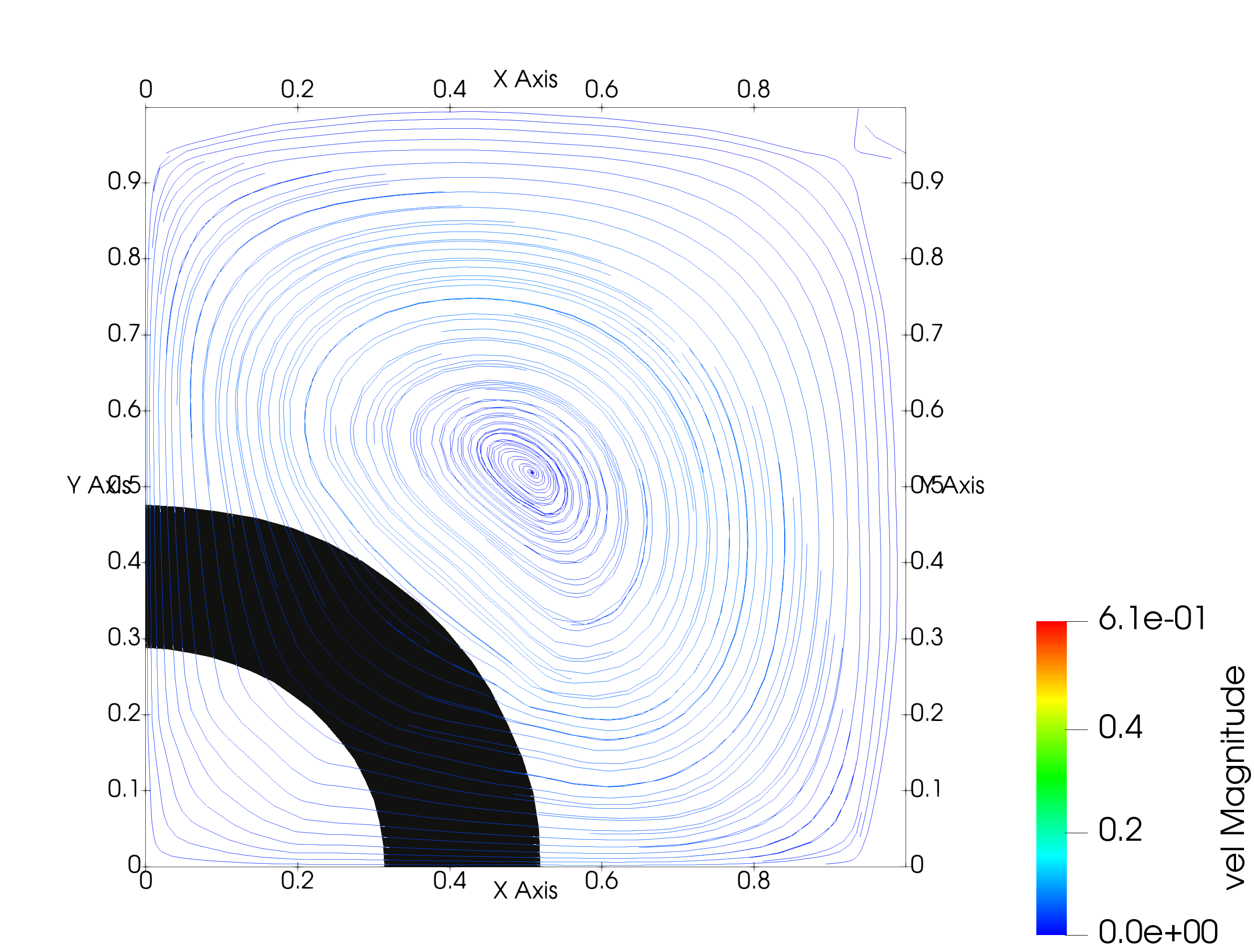}
        \caption{The evolution of an initially deformed ring-shaped structure
                 (computation performed on a quarter of a square for symmetry
                 reasons)}
        \label{fg:annulus}
\end{figure}

\begin{figure}
\centering
\includegraphics[scale = .28]{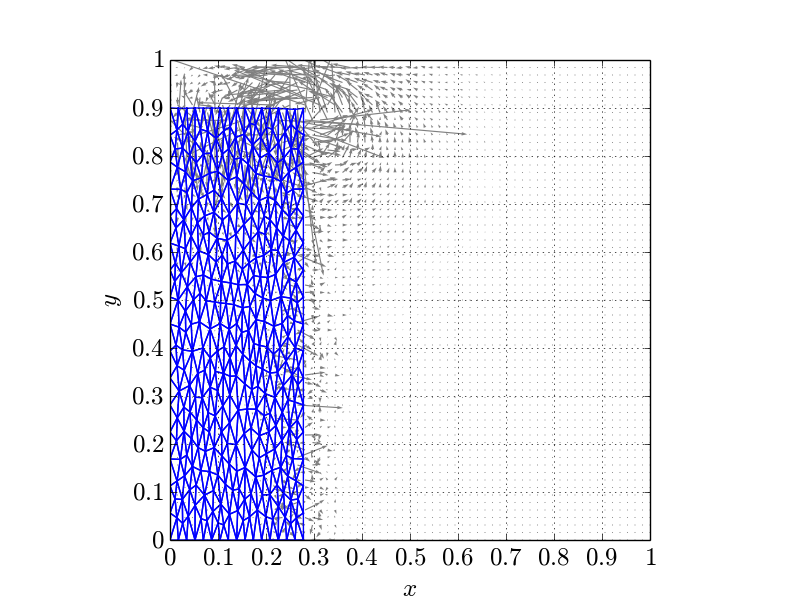}
\includegraphics[scale = .28]{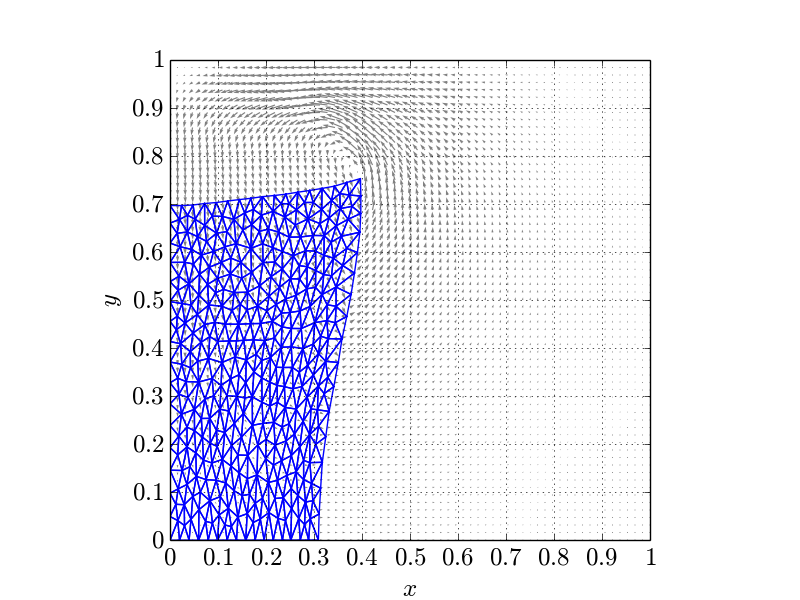}
\includegraphics[scale = .28]{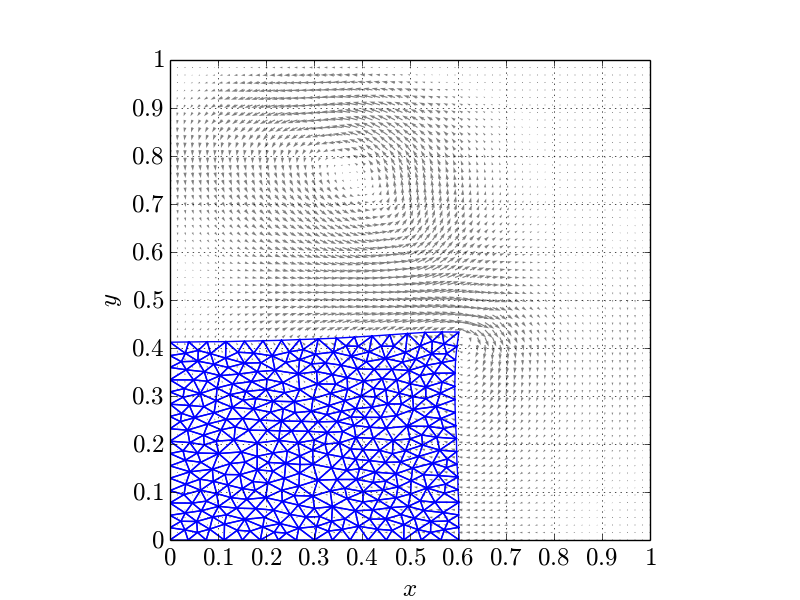}
\includegraphics[scale = .28]{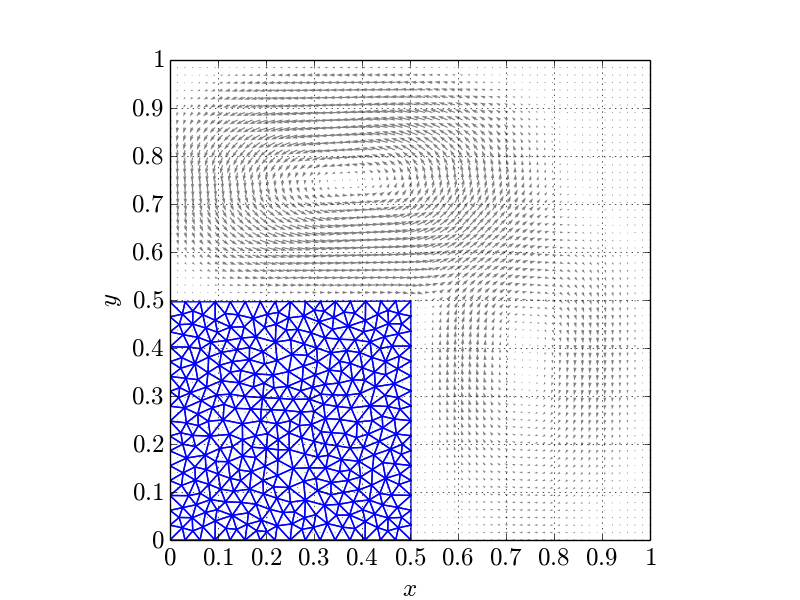}
\caption{Evolution of an initially deformed square structure immersed in a
fluid}
\label{fg:rectangulus}
\end{figure}

\bibliographystyle{plain}
\bibliography{ref}

\begin{thebibliography}{10}

\bibitem{XFEM}
Fr\'{e}d\'{e}ric Alauzet, Benoit Fabr\`eges, Miguel~A. Fern\'{a}ndez, and Mikel
  Landajuela.
\newblock Nitsche-{XFEM} for the coupling of an incompressible fluid with
  immersed thin-walled structures.
\newblock {\em Comput. Methods Appl. Mech. Engrg.}, 301:300--335, 2016.

\bibitem{risulti}
Ferdinando Auricchio, Daniele Boffi, Lucia Gastaldi, Adrien Lefieux, and
  Alessandro Reali.
\newblock On a fictitious domain method with distributed {L}agrange multiplier
  for interface problems.
\newblock {\em Appl. Numer. Math.}, 95:36--50, 2015.

\bibitem{Beirao}
H.~Beir\~{a}o~da Veiga.
\newblock On the existence of strong solutions to a coupled fluid-structure
  evolution problem.
\newblock {\em J. Math. Fluid Mech.}, 6(1):21--52, 2004.

\bibitem{bbf}
D.~Boffi, F.~Brezzi, and M.~Fortin.
\newblock {\em Mixed Finite Element Methods and Applications}, volume~44 of
  {\em Springer Series in Computational Mathematics}.
\newblock Springer-Verlag, New York, 2013.

\bibitem{budini}
D.~Boffi, N.~Cavallini, F.~Gardini, and L.~Gastaldi.
\newblock Local mass conservation of {S}tokes finite elements.
\newblock {\em J. Sci. Comput.}, 52(2):383--400, 2012.

\bibitem{BCG}
D.~Boffi, N.~Cavallini, and L.~Gastaldi.
\newblock The finite element immersed boundary method with distributed
  {L}agrange multiplier.
\newblock {\em SIAM J. Numer. Anal.}, 53(6):2584--2604, 2015.

\bibitem{BGHP}
D.~Boffi, L.~Gastaldi, L.~Heltai, and C.~S. Peskin.
\newblock On the hyper-elastic formulation of the immersed boundary method.
\newblock {\em Comput. Methods Appl. Mech. Engrg.}, 197(25-28):2210--2231,
  2008.

\bibitem{ecmi}
Daniele Boffi, Nicola Cavallini, and Lucia Gastaldi.
\newblock Advances in the mathematical theory of the finite element immersed
  boundary method.
\newblock In Giovanni Russo, Vincenzo Capasso, Giuseppe Nicosia, and Vittorio
  Romano, editors, {\em Progress in Industrial Mathematics at ECMI 2014}, pages
  303--310, Cham, 2016. Springer International Publishing.

\bibitem{BG}
Daniele Boffi and Lucia Gastaldi.
\newblock A fictitious domain approach with {L}agrange multiplier for
  fluid-structure interactions.
\newblock {\em Numer. Math.}, 135(3):711--732, 2017.

\bibitem{BGarXiv}
Daniele Boffi and Lucia Gastaldi.
\newblock A fictitious domain approach with {L}agrange multiplier for
  fluid-structure interactions.
\newblock arXiv:1510.06856v2 [math.NA], 2017.

\bibitem{existence}
Daniele Boffi and Lucia Gastaldi.
\newblock On the existence and the uniqueness of the solution to a
  fluid-structure interaction problem.
\newblock Submitted. arXiv:2006.10536 [math.AP], 2020.

\bibitem{CFL}
Daniele Boffi, Lucia Gastaldi, and Luca Heltai.
\newblock Numerical stability of the finite element immersed boundary method.
\newblock {\em Math. Models Methods Appl. Sci.}, 17(10):1479--1505, 2007.

\bibitem{compressible}
Daniele Boffi, Lucia Gastaldi, and Luca Heltai.
\newblock A distributed {L}agrange formulation of the finite element immersed
  boundary method for fluids interacting with compressible solids.
\newblock In {\em Mathematical and numerical modeling of the cardiovascular
  system and applications}, volume~16 of {\em SEMA SIMAI Springer Ser.}, pages
  1--21. Springer, Cham, 2018.

\bibitem{ruggeri}
Daniele Boffi, Lucia Gastaldi, and Michele Ruggeri.
\newblock Mixed formulation for interface problems with distributed {L}agrange
  multiplier.
\newblock {\em Comput. Math. Appl.}, 68(12, part B):2151--2166, 2014.

\bibitem{wolf}
Daniele Boffi, Lucia Gastaldi, and Sebastian Wolf.
\newblock Higher-order time-stepping schemes for fluid-structure interaction
  problems.
\newblock {\em Discrete Contin. Dyn. Syst. Ser. B}, 25(10):3807--3830, 2020.

\bibitem{Boulakia2017}
M.~Boulakia and S.~Guerrero.
\newblock On the interaction problem between a compressible fluid and a
  {S}aint-{V}enant {K}irchhoff elastic structure.
\newblock {\em Adv. Differential Equations}, 22(1-2):1--48, 2017.

\bibitem{Boulakia2019}
M.~Boulakia, S.~Guerrero, and T.~Takahashi.
\newblock Well-posedness for the coupling between a viscous incompressible
  fluid and an elastic structure.
\newblock {\em Nonlinearity}, 32(10):3548--3592, 2019.

\bibitem{nitsche}
Erik Burman and Miguel~A. Fern\'{a}ndez.
\newblock An unfitted {N}itsche method for incompressible fluid-structure
  interaction using overlapping meshes.
\newblock {\em Comput. Methods Appl. Mech. Engrg.}, 279:497--514, 2014.

\bibitem{causin}
P.~Causin, J.~F. Gerbeau, and F.~Nobile.
\newblock Added-mass effect in the design of partitioned algorithms for
  fluid-structure problems.
\newblock {\em Comput. Methods Appl. Mech. Engrg.}, 194(42-44):4506--4527,
  2005.

\bibitem{CDEG}
A.~Chambolle, B.~Desjardins, M.~J. Esteban, and C.~Grandmont.
\newblock Existence of weak solutions for the unsteady interaction of a viscous
  fluid with an elastic plate.
\newblock {\em J. Math. Fluid Mech.}, 7(3):368--404, 2005.

\bibitem{levelset}
Y.~C. Chang, T.~Y. Hou, B.~Merriman, and S.~Osher.
\newblock A level set formulation of {E}ulerian interface capturing methods for
  incompressible fluid flows.
\newblock {\em J. Comput. Phys.}, 124(2):449--464, 1996.

\bibitem{CSMT}
C.~Conca, J.~San Mart\'{\i}n~H., and M.~Tucsnak.
\newblock Existence of solutions for the equations modelling the motion of a
  rigid body in a viscous fluid.
\newblock {\em Comm. Partial Differential Equations}, 25(5-6):1019--1042, 2000.

\bibitem{Coutand2005}
D.~Coutand and S.~Shkoller.
\newblock Motion of an elastic solid inside an incompressible viscous fluid.
\newblock {\em Arch. Ration. Mech. Anal.}, 176(1):25--102, 2005.

\bibitem{Coutand2006}
D.~Coutand and S.~Shkoller.
\newblock The interaction between quasilinear elastodynamics and the
  {N}avier-{S}tokes equations.
\newblock {\em Arch. Ration. Mech. Anal.}, 179(3):303--352, 2006.

\bibitem{DE1999}
B.~Desjardins and M.~J. Esteban.
\newblock Existence of weak solutions for the motion of rigid bodies in a
  viscous fluid.
\newblock {\em Arch. Ration. Mech. Anal.}, 146(1):59--71, 1999.

\bibitem{DE2000}
B.~Desjardins and M.~J. Esteban.
\newblock On weak solutions for fluid-rigid structure interaction: compressible
  and incompressible models.
\newblock {\em Comm. Partial Differential Equations}, 25(7-8):1399--1413, 2000.

\bibitem{DEGLT}
B.~Desjardins, M.~J. Esteban, C.~Grandmont, and P.~Le~Tallec.
\newblock Weak solutions for a fluid-elastic structure interaction model.
\newblock {\em Rev. Mat. Complut.}, 14(2):523--538, 2001.

\bibitem{Donea1977}
J.~Donea, P.~Fasoli-Stella, and S.~Giuliani.
\newblock Lagrangian and eulerian finite element techniques for transient
  fluid-structure interaction problems.
\newblock {\em Therm and Fluid/Struct Dyn Anal}, B, 1977.
\newblock cited By 0.

\bibitem{doneahuerta2004}
Jean Donea, Antonio Huerta, J.-Ph. Ponthot, and A.~Rodríguez-Ferran.
\newblock {\em Arbitrary Lagrangian–Eulerian Methods}.
\newblock John Wiley \& Sons, Ltd, 2004.

\bibitem{F}
E.~Feireisl.
\newblock On the motion of rigid bodies in a viscous compressible fluid.
\newblock {\em Arch. Ration. Mech. Anal.}, 167(4):281--308, 2003.

\bibitem{girglo1995}
V.~Girault and R.~Glowinski.
\newblock Error analysis of a fictitious domain method applied to a {D}irichlet
  problem.
\newblock {\em Japan J. Indust. Appl. Math.}, 12(3):487--514, 1995.

\bibitem{girglopan}
V.~Girault, R.~Glowinski, and T.-W. Pan.
\newblock A fictitious-domain method with distributed multiplier for the
  {S}tokes problem.
\newblock In {\em Applied nonlinear analysis}, pages 159--174. Kluwer/Plenum,
  New York, 1999.

\bibitem{glopanhj}
R.~Glowinski, T.-W. Pan, T.I. Hesla, and D.D. Joseph.
\newblock A distributed {L}agrange multiplier/fictitious domain method for
  particulate flows.
\newblock {\em International Journal of Multiphase Flow}, 25(5):755 -- 794,
  1999.

\bibitem{glopanper2}
R.~Glowinski, T.-W. Pan, and J.~P{\'e}riaux.
\newblock A fictitious domain method for {D}irichlet problem and applications.
\newblock {\em Comput. Methods Appl. Mech. Engrg.}, 111(3-4):283--303, 1994.

\bibitem{glopanper1}
R.~Glowinski, T.-W. Pan, and J.~P{\'e}riaux.
\newblock A fictitious domain method for external incompressible viscous flow
  modeled by {N}avier-{S}tokes equations.
\newblock {\em Comput. Methods Appl. Mech. Engrg.}, 112(1-4):133--148, 1994.
\newblock Finite element methods in large-scale computational fluid dynamics
  (Minneapolis, MN, 1992).

\bibitem{GM}
C.~Grandmont and Y.~Maday.
\newblock Existence for an unsteady fluid-structure interaction problem.
\newblock {\em M2AN Math. Model. Numer. Anal.}, 34(3):609--636, 2000.

\bibitem{GLS}
M.~D. Gunzburger, H.-C. Lee, and G.~A. Seregin.
\newblock Global existence of weak solutions for viscous incompressible flows
  around a moving rigid body in three dimensions.
\newblock {\em J. Math. Fluid Mech.}, 2(3):219--266, 2000.

\bibitem{hirt}
C.~W. Hirt, A.~A. Amsden, and J.~L. Cook.
\newblock An arbitrary {L}agrangian-{E}ulerian computing method for all flow
  speeds [{J}. {C}omput. {P}hys. {\bf 14} (1974), no. 3, 227--253].
\newblock volume 135, pages 198--216. 1997.

\bibitem{HS}
K.-H. Hoffmann and V.~N. Starovoitov.
\newblock On a motion of a solid body in a viscous fluid. {T}wo-dimensional
  case.
\newblock {\em Adv. Math. Sci. Appl.}, 9(2):633--648, 1999.

\bibitem{HLZ}
Thomas J.~R. Hughes, Wing~Kam Liu, and Thomas~K. Zimmermann.
\newblock Lagrangian-{E}ulerian finite element formulation for incompressible
  viscous flows.
\newblock {\em Comput. Methods Appl. Mech. Engrg.}, 29(3):329--349, 1981.

\bibitem{LM}
J.-L. Lions and E.~Magenes.
\newblock {\em Non-homogeneous boundary value problems and applications. {V}ol.
  {I}}.
\newblock Springer-Verlag, New York-Heidelberg, 1972.

\bibitem{MC2013b}
B.~Muha and S.~\v{C}ani\'{c}.
\newblock Existence of a weak solution to a nonlinear fluid-structure
  interaction problem modeling the flow of an incompressible, viscous fluid in
  a cylinder with deformable walls.
\newblock {\em Arch. Ration. Mech. Anal.}, 207(3):919--968, 2013.

\bibitem{MC2016}
B.~Muha and S.~\v{C}ani\'{c}.
\newblock Existence of a weak solution to a fluid-elastic structure interaction
  problem with the {N}avier slip boundary condition.
\newblock {\em J. Differential Equations}, 260(12):8550--8589, 2016.

\bibitem{nfgk}
Elijah~P. Newren, Aaron~L. Fogelson, Robert~D. Guy, and Robert~M. Kirby.
\newblock Unconditionally stable discretizations of the immersed boundary
  equations.
\newblock {\em J. Comput. Phys.}, 222(2):702--719, 2007.

\bibitem{peskin}
C.~S. Peskin.
\newblock The immersed boundary method.
\newblock {\em Acta Numer.}, 11:479--517, 2002.

\bibitem{RaymondVanni2014}
J.-P. Raymond and M.~Vanninathan.
\newblock A fluid-structure model coupling the {N}avier-{S}tokes equations and
  the {L}am\'{e} system.
\newblock {\em J. Math. Pures Appl. (9)}, 102(3):546--596, 2014.

\bibitem{Serre}
D.~Serre.
\newblock Chute libre d'un solide dans un fluide visqueux incompressible.
  {E}xistence.
\newblock {\em Japan J. Appl. Math.}, 4(1):99--110, 1987.

\bibitem{Ta}
T.~Takahashi.
\newblock Analysis of strong solutions for the equations modeling the motion of
  a rigid-fluid system in a bounded domain.
\newblock {\em Adv. Differential Equations}, 8(12):1499--1532, 2003.

\bibitem{TaTu}
T.~Takahashi and M.~Tucsnak.
\newblock Global strong solutions for the two-dimensional motion of an infinite
  cylinder in a viscous fluid.
\newblock {\em J. Math. Fluid Mech.}, 6(1):53--77, 2004.

\bibitem{temam}
R.~Temam.
\newblock {\em Navier-{S}tokes equations}, volume~2 of {\em Studies in
  Mathematics and its Applications}.
\newblock North-Holland Publishing Co., Amsterdam-New York, revised edition,
  1979.
\newblock Theory and numerical analysis, With an appendix by F. Thomasset.

\bibitem{XZ}
Jinchao Xu and Ludmil Zikatanov.
\newblock Some observations on {B}abu\v ska and {B}rezzi theories.
\newblock {\em Numer. Math.}, 94(1):195--202, 2003.

\bibitem{yu}
Z.~Yu.
\newblock A {DLM/FD} method for fluid/flexible-body interactions.
\newblock {\em Journal of Computational Physics}, 207(1):1 -- 27, 2005.

\end{thebibliography}

\end{document}